\documentclass[12pt]{article}

\usepackage[margin=1in]{geometry}  
\usepackage{graphicx}
\usepackage{tikz}            
\usepackage{amsmath}               
\usepackage{amsthm}                
\usepackage[ansinew]{inputenc}
\usepackage{array,color,amscd}
\usepackage{amsxtra}
\usepackage{amstext}
\usepackage{latexsym}
\usepackage{dsfont}
\usepackage[all]{xy}
\usepackage{amsmath,amscd,amsfonts}
\usepackage{cite}
\usetikzlibrary{matrix,arrows,decorations.pathmorphing}

\newtheorem{thm}{Theorem}[section]
\newtheorem{cor}[thm]{Corollary}
\newtheorem{lem}[thm]{Lemma}
\newtheorem{prop}[thm]{Proposition}
\theoremstyle{definition}

\theoremstyle{remark}

\theoremstyle{conjecture}

\theoremstyle{example}

\numberwithin{equation}{section}

\begin{document}

\title{On Samelson products in $SU(n)$-and $Sp(n)$}
\author{Sajjad Mohammadi}

\date{}
\maketitle
\begin{abstract}
Let $a$ and $b$ be two positive integers such that $a, b < n$. We denote the inclusion $\Sigma \mathbb{C}P^a\rightarrow SU(n)$  by $\varepsilon_{a,n}$. Also, let $m$ and $n$ be two positive integers such that $m < n$. This article has two parts. In the first part, we will study the order of the Samelson product $\langle \varepsilon_{a,n}, \varepsilon_{b,n}\rangle$ where $a+b=n+k$, for $k \geq 0$. Also, we will give two applications. In the second part, we will study the order of the Samelson product $S^{4m-1}\wedge Q_{n-m+1}\rightarrow Sp(n)$, where $Q_{n-m+1}$ is the symplectic quasi-projective space of rank $n-m+1$.
\end{abstract}
$2010$ Mathematics Subject Classification:  $55Q15; 55P10$
\section{Introduction }\label{a0}
Let $Z$ be an $H$-group and $f_1\colon X \rightarrow Z$ and $f_2\colon Y \rightarrow Z$ two maps. Consider the map
$X \times Y \overset{c}\longrightarrow Z$ defined by the commutator
$(x,y)\mapsto f_1(x)f_2(y){f_1(x)}^{-1}{f_2(y)}^{-1}$. This is null homotopic when restricted to $X \vee Y$ so induces a map $\langle f_1, f_2\rangle \colon X \wedge Y \rightarrow Z$. Since the sequence $[X \vee Y, Z]\rightarrow [X \times Y, Z]\rightarrow [X \wedge Y, Z]$ splits, the homotopy class of $\langle f_1, f_2\rangle$ is uniquely determined by those of $f_1$ and $f_2$ and is called the Samelson product of $f_1$ and $f_2$. Now, suppose $Z=\Omega C$ and ${f_1}'\colon SX \rightarrow C$, ${f_2}'\colon SY \rightarrow C$. Let $f_1 \colon X \rightarrow Z$ and $f_2 \colon Y \rightarrow Z$ be the adjoints of ${f_1}'$ and ${f_2}'$, respectively. Then the adjoint of the Samelson product $\langle f_1, f_2\rangle$ is called the Whitehead product and written $[f_1, f_2]\colon S(X \wedge Y )\rightarrow C$.\par
Let $G$ be a simple compact connected Lie group. The calculation of Samelson products plays an important role in classifying the homotopy types of gauge groups of principal $G$-bundles. Also, they are fundamental in studying the homotopy nilpotency and the homotopy commutativity of Lie groups. Despite their topology being studied intensively for decades, surprisingly little is known about the behaviour of commutators on classical Lie groups. The role of the commutator is played by Samelson products. Focusing on the groups $Sp(n)$, in the 1960s Bott \cite{[B]} gave some information on the orders of Samelson products of maps representing the integral homotopy groups. In the 1980s McGibbon \cite{[M]} classified at what primes and for which $n$ the group $Sp(n)$ is homotopy commutative (and so has all of its Samelson products trivial). In the late 2000s, Hamanaka, Kaji and Kono \cite{[HKK]} considered the map $\epsilon_n\colon Q_n \rightarrow Sp(n)$, where $Q_n$ is James quasi-projective space, and has the property that $H_*(Sp(n)) \cong \Lambda({\tilde{H}}_*(Sp(n)))$. They determined the order of $\langle \epsilon_2, \epsilon_2\rangle$ at the prime three. In following some related work by them, Nagao \cite{[N]} showed that the Samelson product of $\epsilon_2$ with itself has order $280$, and determined the odd primary components of the orders of the Samelson products of $\epsilon_n$ with itself for $n\in \{3, 4, 5\}$. Moreover, Samelson products have been studied extensively for the classical groups. In \cite{[H1]}, Hamanaka studies the Samelson products of $U(n)$ localized at a prime $p$. In \cite{[O]}, Oshima determines the Samelson product
$$\pi_n(G_2)\times \pi_{11}(G_2)\rightarrow \pi_{n+11}(G_2), $$
for $n=3, 11$. Also in \cite{[KKKT]}, Kamiyama, Kishimoto, Kono and Tsukuda have studied the Samelson products of $SO(3)$. As we mention, there are few results on Samelson products in Special unitary groups. Since this subject is valuable and has many applications, in this article, we will study this subject.\par
Let $a$ and $b$ two positive integers such that $a, b < n$. We denote the inclusion $\Sigma \mathbb{C}P^a\rightarrow SU(n)$ by $\varepsilon_{a,n}$. The Samelson product $\Sigma \mathbb{C}P^a \wedge \Sigma \mathbb{C}P^b \rightarrow U(n)$ has been studied by Hamanaka and Kono in \cite{[H3]}, where $a+b=n-1$. In the first part of this article, we extend this work by studying the order of the Samelson product $\langle \varepsilon_{a,n}, \varepsilon_{b,n}\rangle \colon \Sigma \mathbb{C}P^a \wedge \Sigma \mathbb{C}P^b \rightarrow SU(n)$ where $a+b=n+k$, for $k\geq 0$. Let $d$ be the order of the Samelson product $\langle \varepsilon_{a,n}, \varepsilon_{b,n}\rangle$. We will prove the following theorems.
\begin{thm}\label{a1}
Let $a+b=n$. Localized at odd primes, the following hold:
\begin{displaymath}
d = \left\{ \begin{array}{ll}
3 & \textrm{\text{if} $a=1, b=2$ },\\
5\cdot 3 & \textrm{\text{if} $ a=1, b=3$ },\\
5\cdot 3 & \textrm {\text{if} $a=b=2$,\quad \text{or}\quad $a=2, b=3$}.
\end{array} \right.
\end{displaymath}
\end{thm}
\begin{thm}\label{a2}
Let $a+b=n+k$. Localized at odd primes, the following hold:
\begin{displaymath}
d = \left\{ \begin{array}{ll}
5\cdot 3 & \textrm{\text{if} $k=1, a=b=2$ },\\
5\cdot 3^2& \textrm{\text{if} $k=1, a=2, b=3$ },\\
7\cdot 5\cdot 3^2 & \textrm{\text{if} $k=1, a=b=3$,\quad \text{or}\quad $k=2, a=b=3$ },\\
7\cdot 5^2\cdot 3^2 & \textrm{\text{if} $k=1, a=2, b=4$ },\\
7\cdot 5^3\cdot 3^4 & \textrm{\text{if} $k=3, a=b=4$}.
\end{array} \right.
\end{displaymath}
\end{thm}
We denote by $-_{(p)}$ the localisation at a prime $p$ in the sense of Bousfield and Kan \cite{[BK]}. We prove these theorems in Sections \ref{b0} and \ref{c0} respectively, and in Section \ref{d0}, we give two applications of them. The first application is on the homotopy commutativity of $SU(n)_{(p)}$, for $n=3, 4, 5$ and $p > 2n$. By \cite{[M]}, we know that this fact is already proved by McGibbon but in this article, we give a new proof of this fact which does not require the homotopy groups of $SU(n)$, for $n=3, 4$ and $5$. Namely, we obtain the following corollary.
\begin{cor}\label{a3}
For $n=3, 4, 5$, if $p > 2n$ then $SU(n)_{(p)}$ is homotopy commutative.
\end{cor}
In the second application, we calculate the order of the commutator map $SU(4)\wedge SU(4) \rightarrow SU(4)$, when localized at odd primes. This gives us an upper bound on the number of homotopy types of gauge groups for principal $SU(4)$-bundles over an $n$-sphere.
\begin{thm}\label{a4}
Localized at odd primes, the order of the commutator map $SU(4)\wedge SU(4)\overset{c}\rightarrow SU(4)$ is $3^2\cdot 5 \cdot 7$.
\end{thm}
Let $m$ and $n$ be two positive integers such that $m < n$. Let $Q_{n-m+1}$ be the symplectic quasi-projective space of rank $n-m+1$, we denote the inclusion $Q_{n-m+1}\rightarrow Sp(n)$ by $\epsilon_{m,n}$. Also, let $\varepsilon_{m,n}\colon S^{4m-1} \rightarrow Sp(n)$ represents the generator of $\pi_{4m-1}(Sp(n))$. In the second part of this article, we will study the order of the Samelson product $\langle \varepsilon_{m,n},\epsilon_{m,n} \rangle\colon S^{4m-1}\wedge Q_{n-m+1}\rightarrow Sp(n)$.  We will prove the following theorem.
\begin{thm}\label{a5}
The order of the Samelson product $\langle \varepsilon_{m,n},\epsilon_{m,n} \rangle\colon S^{4m-1}\wedge Q_{n-m+1}\rightarrow Sp(n) $ is equal to
\begin{displaymath}
 \left\{ \begin{array}{ll}
\frac{(2n+1)!}{(2n-2m+1)!} & \textrm{if $m$ is even },\\\\
\frac{2(2n+1)!}{(2n-2m+1)!} & \textrm{if $m$ is odd }.\\
\end{array} \right.
\end{displaymath}
\end{thm}
The theorem recovers the known case in \cite{[KK]} when $m=1$.
\section{The Samelson product $\langle \varepsilon_{a,n}, \varepsilon_{b,n}\rangle$ when $a+b = n$}\label{b0}
In this section, we by use of unstable $K$-theory will calculate the order of the Samelson
product $\langle \varepsilon_{a,n}, \varepsilon_{b,n}\rangle$ when $a+b=n$ in some cases for $a$ and $b$. We denote the homogeneous space $U(\infty)/U(n)$ by $W_n$ and $[X, U(n)]$ by $U_n(X)$. Put $X=\Sigma \mathbb{C}P^a \wedge \Sigma \mathbb{C}P^b$, where $a+b=n$. Recall that the cohomologies of $U(n)$, $U(\infty)$ and $BU(\infty)$ are given by
$H^*(U(n);\mathbb{Z})= \Lambda (x_1,x_3,\cdots,x_{2n-1})$, $H^*(U(\infty);\mathbb{Z})= \Lambda (x_1,x_3,\cdots)$ and $H^*(BU(\infty);\mathbb{Z})= \mathbb{Z}[c_1,c_2,\cdots]$, where $x_{2i-1}=\sigma c_i$ and $\sigma$ is the cohomology suspension and also $c_i$ is the $i$-th universal Chern class and $|x_{2i-1}|=2i-1$. Consider the projection $\pi\colon U(\infty) \rightarrow W_n$. As an algebra we have that the cohomology of $W_n$ is given by $H^*(W_n;\mathbb{Z})=\Lambda (\bar{x}_{2n+1}, \bar{x}_{2n+3}, \cdots)$, where $\pi^*(\bar{x}_{2i+1})=x_{2i+1}$. Also, for $\ast< 4n$, we have $H^*(\Omega W_n;\mathbb{Z})= \mathbb{Z} <a_{2n},a_{2n+2},\cdots > $, that is, $H^*(\Omega W_n;\mathbb{Z})$ is a free $\mathbb{Z}$-module with its generators $a_{2i} (i\geq n)$, where $a_{2i} =\sigma (\bar{x}_{2i+1})$. Consider the fibration sequence
\begin{eqnarray}
\Omega U(\infty)\overset{\Omega \pi}\longrightarrow \Omega W_n\overset{\delta}\longrightarrow U(n)
\overset{j}\longrightarrow U(\infty)\overset{\pi}\longrightarrow W_n.
\end{eqnarray}
Now apply the functor $[X,-]$ to fibration $(2.1)$. Then there is an exact sequence of groups
\begin{eqnarray*}
[X,\Omega U(\infty)]\overset{(\Omega \pi)_*}\longrightarrow [X,\Omega W_n]\overset{\delta_*}\longrightarrow [X,
U(n)]\overset{j_*}\longrightarrow [X,U(\infty)]\overset{\pi_*}\longrightarrow [X, W_n].
\end{eqnarray*}
Since $W_n$ is $2n$-connected, for $i \leq 2n$ we have $\pi_i(W_n)=0$. By the homotopy sequence of the fibration $(2.1)$, we have the following isomorphisms (for example, see \cite{[H2]})
\begin{eqnarray*}
\pi_{2n+1}(W_n)\cong \mathbb{Z},\quad
\pi_{2n+2}(W_n)\cong \left\{ \begin{array}{ll}
0 & \textrm{\text{if $n$ is odd} },\\
\mathbb{Z}_2 & \textrm{\text{if $n$ is even} },
\end{array} \right.\quad
\pi_{2n+3}(W_n)\cong \left\{ \begin{array}{ll}
\mathbb{Z} & \textrm{\text{if $n$ is odd} },\\
\mathbb{Z}\oplus\mathbb{Z}_2 & \textrm{\text{if $n$ is even} }.
\end{array} \right.
\end{eqnarray*}
Observe that $[X,U(\infty)]\cong [\Sigma X, BU(\infty)]\cong \tilde{K}^0(\Sigma X)\cong 0$, since $\Sigma X$ is a $CW$-complex consisting only of odd dimensional cells. Thus we get the following exact sequence:
$$\tilde{K}^0(X)\overset{(\Omega \pi)_*}\longrightarrow [X,\Omega W_n]\overset{\delta_*}\longrightarrow U_n(X)\rightarrow 0.$$
Therefore we have the following lemma.
\begin{lem}\label{b1}
$U_n(X)\cong Coker (\Omega \pi)_*$. $\quad\Box$
\end{lem}
Let $a_{2n}$ and $a_{2n+2}$ be generators of $H^{2n}(\Omega W_n)\cong \mathbb{Z}$ and $H^{2n+2}(\Omega W_n)\cong \mathbb{Z}$ respectively, and $\alpha\in [X, \Omega W_n]$. We define a map
$\lambda\colon [X,\Omega W_n]\rightarrow H^{2n}(X)\oplus H^{2n+2}(X)\cong \mathbb{Z}^{\oplus3}$, by $\lambda(\alpha)=(\alpha^*(a_{2n}), \alpha^*(a_{2n+2}))$. The cohomology class $\bar{x}_{2n+1}$ represents a map $\bar{x}_{2n+1}\colon W_n \rightarrow K(\mathbb{Z}, 2n+1)$ and $a_{2n}=\sigma (\bar{x}_{2n+1})$ represents its loop $\Omega \bar{x}_{2n+1}\colon  \Omega W_n \rightarrow \Omega K(\mathbb{Z}, 2n+1)$. Similarly $a_{2n+2}=\sigma(\bar{x}_{2n+3})$ represents a loop map. This implies the map $\lambda$ is a group homomorphism. Recall that $H^*(\mathbb{C}P^n)= \mathbb{Z}[t]/(t^{n+1})$, where $|t|=2$ and $K(\mathbb{C}P^n)= \mathbb{Z}[x]/(x^{n+1})$. Note that $\tilde{K}^0(\Sigma \mathbb{C}P^a \wedge \Sigma \mathbb{C}P^b)$ is a free abelian group generated by $\zeta_1\otimes (x^i \otimes x^j)$, where
$1 \leq i \leq a, 1 \leq j \leq b$ and $\zeta_1$ a generator of $\tilde{K}^0(S^2)$. In the following, we consider some cases for $a$ and $b$.\\
$\bullet\quad a=b=2$\\
Consider the following cofibration sequence
\begin{eqnarray}
S^3\overset{\eta}\longrightarrow S^2\rightarrow \mathbb{C}P^2\overset{q}\longrightarrow S^4 \overset{\Sigma\eta}\longrightarrow S^3,
\end{eqnarray}
where $\eta$ and $q$ are the Hopf map and the quotient map, respectively. We have the following lemma.
\begin{lem}\label{b2}
The map $\lambda\colon [X,\Omega W_4]\rightarrow H^8(X)\oplus H^{10}(X)$ is monomorphism.
\end{lem}
\begin{proof}
First, we show the group $[X,\Omega W_4]$ is a free abelian group. Recall $X=\Sigma \mathbb{C}P^2 \wedge \Sigma \mathbb{C}P^2$. Consider the following cofibration
\begin{eqnarray}
S^4\rightarrow S^3\rightarrow \Sigma\mathbb{C}P^2\rightarrow S^5 \rightarrow S^4.
\end{eqnarray}
Now apply $[\Sigma\mathbb{C}P^2\wedge -,\Omega W_4]$ to the cofibration $(2.3)$, we get
\begin{align*}
[\Sigma^5\mathbb{C}P^2,\Omega W_4]\rightarrow [\Sigma^6\mathbb{C}P^2,\Omega W_4]&\rightarrow [\Sigma\mathbb{C}P^2\wedge\Sigma\mathbb{C}P^2,\Omega W_4]\\&
\rightarrow [\Sigma^4\mathbb{C}P^2,\Omega W_4]\rightarrow [\Sigma^5\mathbb{C}P^2,\Omega W_4].\quad (\ast)
\end{align*}
Now apply $[\Sigma^4-,\Omega W_4]$, $[\Sigma^5-,\Omega W_4]$ and $[\Sigma^6-,\Omega W_4]$ to the cofibration $(2.2)$, respectively
\begin{align*}
&\pi_8(W_4)\rightarrow \pi_9(W_4)\rightarrow [\Sigma^4\mathbb{C}P^2,\Omega W_4]\rightarrow \pi_7(W_4),\\
&\pi_9(W_4)\overset{\eta'}\rightarrow \pi_{10}(W_4)\rightarrow [\Sigma^5\mathbb{C}P^2,\Omega W_4]\rightarrow \pi_8(W_4),\\
&\pi_{10}(W_4)\overset{\eta''}\rightarrow \pi_{11}(W_4)\rightarrow [\Sigma^6\mathbb{C}P^2,\Omega W_4]\rightarrow \pi_9(W_4)\rightarrow \pi_{10}(W_4),
\end{align*}
where $\eta'$ and $\eta''$ are induced by the Hopf maps $\Sigma^7\eta \colon S^{10}\rightarrow S^9$ and $\Sigma^8\eta \colon S^{11}\rightarrow S^{10}$. The inclusion $i\colon S^9 \rightarrow W_4$ generates $\pi_9(W_4)\cong \mathbb{Z}$ and the compositions $j'\colon S^{10} \overset{\Sigma^7\eta}\longrightarrow S^9 \overset{i}\longrightarrow W_4$ and $j''\colon S^{11} \overset{\Sigma^8\eta}\longrightarrow S^{10} \overset{\Sigma^7\eta}\longrightarrow S^9 \overset{i}\longrightarrow W_4$ generate $\pi_{10}(W_4)\cong \mathbb{Z}_2$
and $\mathbb{Z}_2$-summand of $\pi_{11}(W_4)\cong \mathbb{Z}\oplus\mathbb{Z}_2$, respectively. Since $\eta'$ sends $i$ to $j'$, so the map $\eta'\colon \mathbb{Z}\rightarrow\mathbb{Z}_2$ is surjective. Also, since $\eta''$ sends $j'$ to $j''$, then the cokernel of $\eta''$ is isomorphic to $\mathbb{Z}$. Thus we have the following exact sequences
\begin{align*}
& 0\rightarrow\mathbb{Z}\rightarrow [\Sigma^4\mathbb{C}P^2,\Omega W_4]\rightarrow 0,\\
& \mathbb{Z}\overset{onto}\longrightarrow\mathbb{Z}_2\rightarrow [\Sigma^5\mathbb{C}P^2,\Omega W_4]\rightarrow 0,\\
&\mathbb{Z}_2\rightarrow \mathbb{Z}\oplus\mathbb{Z}_2\rightarrow [\Sigma^6\mathbb{C}P^2,\Omega W_4]\rightarrow \mathbb{Z}\overset{onto}\longrightarrow\mathbb{Z}_2,
\end{align*}
and therefore we get $[\Sigma^4\mathbb{C}P^2,\Omega W_4]\cong \mathbb{Z}$, $[\Sigma^5\mathbb{C}P^2,\Omega W_4]\cong 0$ and $[\Sigma^6\mathbb{C}P^2,\Omega W_4]\cong \mathbb{Z}\oplus \mathbb{Z}$, respectively. Thus according to the long exact sequence $(\ast)$ we get the following short exact
sequence
$$0\rightarrow \mathbb{Z}\oplus\mathbb{Z}\rightarrow [\Sigma\mathbb{C}P^2\wedge\Sigma\mathbb{C}P^2,\Omega W_4]\rightarrow \mathbb{Z}\rightarrow 0,$$
therefore we get $[\Sigma\mathbb{C}P^2\wedge\Sigma\mathbb{C}P^2,\Omega W_4]$ is a free abelian group isomorphic to $\mathbb{Z}\oplus\mathbb{Z}\oplus\mathbb{Z}$. Thus the homomorphism $\lambda$ is monomorphism.
\end{proof}
We have the following lemma.
\begin{lem}\label{b3}
Im $\lambda\circ (\Omega\pi)_*$ is generated by $\alpha_{1,1},\cdots, \alpha_{2,2}$, where
$$\alpha_{1,1}=(\frac{1}{2}4!, \frac{1}{2}4!, \frac{1}{4}5!), \alpha_{1,2}=(4!,0, \frac{1}{2}5!), \alpha_{2,1}=(0, 4!, \frac{1}{2}5!), \alpha_{2,2}=(0,0,5!). $$
\end{lem}
\begin{proof}
By \cite{[H3]}, recall that $(\Omega\pi)_*(a_{2n})= n!ch_n$. Also, $H^*(\mathbb{C}P^2)= \mathbb{Z}[t]/(t^3)$, where $|t|=2$ and $K(\mathbb{C}P^2)= \mathbb{Z}[x]/(x^3)$, where $chx=t+\frac{1}{2}t^2$ and $chx^2 = t^2$. Note that $\tilde{K}^0(\Sigma\mathbb{C}P^2\wedge\Sigma\mathbb{C}P^2)\cong \mathbb{Z}^{\oplus4}$ with Chern characters
$ch(x\otimes x)=t\otimes t+ t\otimes\frac{1}{2}t^2+ \frac{1}{2}t^2\otimes t+\frac{1}{2}t^2\otimes\frac{1}{2}t^2$ and
\begin{align*}
ch(x\otimes x^2)=t\otimes t^2+ \frac{1}{2}t^2\otimes t^2,\quad
ch(x^2\otimes x)=t^2\otimes t+ t^2\otimes\frac{1}{2}t^2, \quad
ch(x^2\otimes x^2)=t^2\otimes t^2.
\end{align*}
According to the definition of the map $\lambda$, we have
$$\lambda \circ(\Omega \pi)_* (\zeta_1\otimes(x\otimes x))=((\Omega \pi\circ (\zeta_1\otimes(x\otimes x)))^*(a_8), (\Omega \pi\circ (\zeta_1\otimes(x\otimes x)))^*(a_{10}) ).$$
The calculation of the first component is as follows
\begin{align*}
(\Omega \pi\circ (\zeta_1\otimes(x\otimes x)))^*(a_8)&=a_8\circ\Omega \pi\circ(\zeta_1\otimes(x\otimes x))=4!ch_4(\zeta_1\otimes(x\otimes x))\\
&=4!ch_1(\zeta_1)ch_3(x\otimes x)=4!\sigma^2(t\otimes\frac{1}{2}t^2+ \frac{1}{2}t^2\otimes t),
\end{align*}
and calculation the second component is as follows
\begin{align*}
(\Omega \pi\circ (\zeta_1\otimes(x\otimes x)))^*(a_{10})&=a_{10}\circ\Omega \pi\circ(\zeta_1\otimes(x\otimes x))=5!ch_5(\zeta_1\otimes(x\otimes x))\\
&=5!ch_1(\zeta_1)ch_4(x\otimes x)=5!\sigma^2(\frac{1}{2}t^2\otimes\frac{1}{2}t^2).
\end{align*}
Therefore we obtain $\alpha_{1,1}=\lambda \circ(\Omega \pi)_* (\zeta_1\otimes(x\otimes x))=(\frac{1}{2}4!, \frac{1}{2}4!, \frac{1}{4}5!)$. Similarly, we can show
that $\alpha_{1,2}=(4!,0, \frac{1}{2}5!)$, $\alpha_{2,1}=(0, 4!, \frac{1}{2}5!)$ and $\alpha_{2,2}=(0,0,5!)$.
\end{proof}
Let $\gamma$ be the commutator map $U(n)\wedge U(n)\rightarrow U(n)$. Since $U(\infty)$ is an infinite loop space, it is homotopy commutative. Therefore $\gamma$ composed to $U(\infty)$ is null homotopic, implying that there is a lift $\tilde{\gamma}\colon U(n)\wedge U(n)\rightarrow \Omega W_n$ such that $\delta \circ \tilde{\gamma} \simeq \gamma$. By \cite{[H3]}, we have the following lemma.
\begin{lem}\label{b4}
We can choose a lift $\tilde{\gamma}$ such that $\tilde{\gamma}^*(a_{2n})=\underset{i+j=n-1}\sum x_{2i+1}\otimes x_{2j+1}.\quad \Box$
\end{lem}
Similarly, by using of the method in \cite{[H3]} we have $\tilde{\gamma}$ so that $\tilde{\gamma}^*(a_{2n+2})= \underset{i+j=n} \sum x_{2i+1}\otimes x_{2j+1}$. Therefore in the case of $n=4$ and $a=b=2$, we get the relation
$$\langle \varepsilon_{2,4}, \varepsilon_{2,4}\rangle=\delta_*(\tilde{\gamma}\circ (\varepsilon_{2,4}\wedge \varepsilon_{2,4}) ).\qquad (\star)$$
We denote the equivalence class of $\tilde{\gamma}\circ (\varepsilon_{2,4}\wedge \varepsilon_{2,4})$ in the Cokernel of $(\Omega\pi)_*$ by $[\tilde{\gamma}\circ (\varepsilon_{2,4}\wedge \varepsilon_{2,4})]$ and the order of an element $x$ of a group by $|x|$. Then by Lemma \ref{b1} and relation $(\star)$, we have
$$|\langle \varepsilon_{2,4}, \varepsilon_{2,4}\rangle|=|[\tilde{\gamma}\circ (\varepsilon_{2,4}\wedge \varepsilon_{2,4})]|.$$
Therefore we will calculate $|[\tilde{\gamma}\circ (\varepsilon_{2,4}\wedge \varepsilon_{2,4})]|$. By Lemma \ref{b3}, we know Im $ \lambda\circ (\Omega\pi)_*$ is generated by $\alpha_{1,1}$, $\alpha_{1,2}$, $\alpha_{2,1}$ and $\alpha_{2,2}$. On the other hand, it follows from Lemma \ref{b4} that $\lambda(\tilde{\gamma}\circ (\varepsilon_{2,4}\wedge \varepsilon_{2,4}))=(1,1,1)$. Now localized at odd prime $p$, we observe that $5(\alpha_{1,1}-\alpha_{1,2})+5(\alpha_{1,1}-\alpha_{2,1})+\alpha_{2,2} = 5\cdot 3(1, 1, 1)$ and no other combination of $\alpha_{i,j}$ gives a smaller multiple of $(1, 1, 1)$. Thus we get the following
proposition.
\begin{prop}\label{b5}
Localized at odd prime $p$, $|[\tilde{\gamma}\circ (\varepsilon_{2,4}\wedge \varepsilon_{2,4})]|= 5\cdot 3$. $\quad\Box$
\end{prop}
$\bullet\quad a=2, b=3$\\
In this case we have the following lemma.
\begin{lem}\label{b6}
The map $\lambda\colon [X,\Omega W_5]\rightarrow H^{10}(X)\oplus H^{12}(X)$ is monomorphism.
\end{lem}
\begin{proof}
Recall $X=\Sigma \mathbb{C}P^2 \wedge \Sigma \mathbb{C}P^3$. Similar to the Lemma \ref{b2}, we show the group $[X,\Omega W_5]$ is a free abelian group. Consider the following cofibration sequence
\begin{eqnarray}
S^6\rightarrow \Sigma\mathbb{C}P^2\rightarrow \Sigma\mathbb{C}P^3\rightarrow S^7.
\end{eqnarray}
Now apply $[\Sigma\mathbb{C}P^2\wedge -,\Omega W_5]$ to the cofibration $(2.4)$, we get
\begin{align*}
[\Sigma\mathbb{C}P^2\wedge \Sigma^2\mathbb{C}P^2 ,\Omega W_5]\rightarrow [\Sigma^8\mathbb{C}P^2,\Omega W_5]&\rightarrow [\Sigma\mathbb{C}P^2\wedge \Sigma\mathbb{C}P^3,\Omega W_5]\\&\rightarrow [\Sigma\mathbb{C}P^2\wedge\Sigma\mathbb{C}P^2,\Omega W_5]
\rightarrow [\Sigma^7\mathbb{C}P^2,\Omega W_5].\quad (\ast\ast)
\end{align*}
Apply $[\Sigma^7-,\Omega W_5]$ and $[\Sigma^8-,\Omega W_5]$ to the cofibration $(2.2)$, respectively, we obtain the group $[\Sigma^7\mathbb{C}P^2,\Omega W_5]$ is zero and the group $[\Sigma^8\mathbb{C}P^2,\Omega W_5]$ is isomorphic to $\mathbb{Z}\oplus \mathbb{Z}$. Also, we can easily show that the group $[\Sigma\mathbb{C}P^2\wedge\Sigma\mathbb{C}P^2,\Omega W_5]$ is isomorphic to $\mathbb{Z}$. Again apply $[\Sigma\mathbb{C}P^2\wedge -,\Omega W_5]$ to the cofibration sequence $S^4 \rightarrow \Sigma^2\mathbb{C}P^2 \rightarrow S^6$, we get the group $[\Sigma\mathbb{C}P^2\wedge \Sigma^2\mathbb{C}P^2,\Omega W_5]$ is zero. Thus according to the $(\ast\ast)$ we get the exact sequence
$$0\rightarrow \mathbb{Z}\oplus \mathbb{Z}\rightarrow [\Sigma\mathbb{C}P^2\wedge \Sigma\mathbb{C}P^3,\Omega W_5]\rightarrow \mathbb{Z}\rightarrow 0,$$
therefore we get $[\Sigma\mathbb{C}P^2\wedge \Sigma\mathbb{C}P^3,\Omega W_5]$ is a free abelian group isomorphic to $\mathbb{Z}\oplus \mathbb{Z}\oplus \mathbb{Z}$.
\end{proof}
Similar to the Lemma \ref{b3}, we get the following lemma.
\begin{lem}\label{b7}
Im $\lambda\circ (\Omega\pi)_*$ is generated by $\alpha_{1,1},\cdots, \alpha_{2,3}$, where
\begin{align*}
&\alpha_{1,1}=(\frac{1}{3!}5!, \frac{1}{4}5!, \frac{1}{2\cdot 3!}6!), && \alpha_{1,2}=(\frac{1}{2}5!,\frac{1}{2}5!,\frac{1}{4}6!), &&&\alpha_{1,3}=(2\cdot 5!,0,6!), \\
&\alpha_{2,1}=(0,\frac{1}{2}5!, \frac{1}{3!}6!), && \alpha_{2,2}=(0,5!,\frac{1}{2}6!), &&&\alpha_{2,3}=(0,0,2\cdot 6!). \quad\Box
\end{align*}
\end{lem}
On the other hand, it follows from Lemma \ref{b4} that $\lambda(\tilde{\gamma}\circ (\varepsilon_{2,5}\wedge \varepsilon_{3,5}))=(1,1,1)$. Localized at odd prime $p$, we have $(3\alpha_{1,1}-\alpha_{1,2})+(\alpha_{2,2}-\alpha_{2,1})+(\alpha_{1,3}-\alpha_{2,3}) = 5\cdot 3(1, 1, 1)$. Thus we get the following proposition.
\begin{prop}\label{b8}
Localized at odd prime $p$, $|[\tilde{\gamma}\circ (\varepsilon_{2,5}\wedge \varepsilon_{3,5})]|= 5\cdot 3$. $\quad\Box$
\end{prop}
$\bullet\quad a=1$\\
In the following cases, we need proof of the following lemma.
\begin{lem}\label{b9}
In cases $n$ is odd and $n$ is even (localized at prime $p$), the map $\lambda\colon [X,\Omega W_n]\rightarrow H^{2n}(X)\oplus H^{2n+2}(X)$ is monic.
\end{lem}
\begin{proof}
Put $X=\Sigma \mathbb{C}P^1 \wedge \Sigma \mathbb{C}P^{n-1}$. We show the group $[X,\Omega W_n]$ is a free abelian group. Equivalently, we show the group $[\Sigma^4 \mathbb{C}P^{n-1},\Omega W_n]$ is a free abelian group. Consider the following cofibration sequence
\begin{eqnarray}
S^{2n-3}\rightarrow \mathbb{C}P^{n-2}\rightarrow \mathbb{C}P^{n-1}\rightarrow S^{2n-2}.
\end{eqnarray}
Now apply $[\Sigma^4-,\Omega W_n]$ to the cofibration $(2.5)$, we get
\begin{align*}
[\Sigma^5\mathbb{C}P^{n-2} ,\Omega W_n]\rightarrow \pi_{2n+2}(\Omega W_n)&\rightarrow [\Sigma^4\mathbb{C}P^{n-1},\Omega W_n]\\&\rightarrow [\Sigma^4\mathbb{C}P^{n-2} ,\Omega W_n]\rightarrow \pi_{2n+1}(\Omega W_n).\quad (\ast\ast\ast)
\end{align*}
First, we calculate the groups $[\Sigma^4\mathbb{C}P^{n-2} ,\Omega W_n]$ and $[\Sigma^5\mathbb{C}P^{n-2} ,\Omega W_n]$. Apply $[\Sigma^4-,\Omega W_n]$ and $[\Sigma^5-,\Omega W_n]$ to the cofibration
$$S^{2n-5}\rightarrow \mathbb{C}P^{n-3}\rightarrow \mathbb{C}P^{n-2}\rightarrow S^{2n-4}, $$
respectively. We obtain the exact sequences
\begin{align*}
&[\Sigma^5\mathbb{C}P^{n-3} ,\Omega W_n]\rightarrow \pi_{2n}(\Omega W_n)\rightarrow [\Sigma^4\mathbb{C}P^{n-2},\Omega W_n]\rightarrow [\Sigma^4\mathbb{C}P^{n-3} ,\Omega W_n]\rightarrow \pi_{2n-1}(\Omega W_n),\\
&[\Sigma^6\mathbb{C}P^{n-3} ,\Omega W_n]\overset{f}\rightarrow \pi_{2n+1}(\Omega W_n)\rightarrow [\Sigma^5\mathbb{C}P^{n-2},\Omega W_n]\rightarrow [\Sigma^5\mathbb{C}P^{n-3} ,\Omega W_n]\rightarrow \pi_{2n}(\Omega W_n)
\end{align*}
respectively. Since $W_n$ is $2n$-connected, we conclude that the terms $[\Sigma^5\mathbb{C}P^{n-3} ,\Omega W_n]$ and $[\Sigma^4\mathbb{C}P^{n-3} ,\Omega W_n]$ are zero. Therefore the group $[\Sigma^4\mathbb{C}P^{n-2},\Omega W_n]$ is isomorphic to $\pi_{2n+1}(W_n)\cong\mathbb{Z}$. When $n$ is odd then the group $[\Sigma^5\mathbb{C}P^{n-2},\Omega W_n]$ is isomorphic to zero. Let $n$ be even. By apply $[\Sigma^6-,\Omega W_n]$ to the cofibration $ \mathbb{C}P^{n-4}\rightarrow \mathbb{C}P^{n-3}\rightarrow S^{2n-6}$, we can conclude that the group $[\Sigma^6\mathbb{C}P^{n-3},\Omega W_n]$ is isomorphic to $\pi_{2n+1}(W_n)\cong\mathbb{Z}$. Therefore we get the following exact sequence
$$\mathbb{Z}\overset{f}\rightarrow \mathbb{Z}_2\rightarrow [\Sigma^5\mathbb{C}P^{n-2},\Omega W_n]\rightarrow 0.$$
Now, since the map $f$ sends $a_1\colon S^{2n+1} \rightarrow W_n$ to $a_2\colon S^{2n+2}\overset{\Sigma^{2n-1}\eta} \longrightarrow S^{2n+1}\overset{a_1}\longrightarrow W_n$, so the map $f$ is surjective. Thus by exactness we can conclude that the group $[\Sigma^5\mathbb{C}P^{n-2},\Omega W_n]$ is also isomorphic to zero.\par
Therefore according to the long exact sequence $(\ast\ast\ast)$, in case $n$ is odd then we get the group $[\Sigma^4\mathbb{C}P^{n-1},\Omega W_n]$ is a free abelian group isomorphic to $\mathbb{Z}\oplus\mathbb{Z}$. When $n$ is even, we obtain the following exact sequence
$$0\rightarrow \mathbb{Z}\oplus \mathbb{Z}_2\rightarrow [\Sigma^4\mathbb{C}P^{n-1},\Omega W_n]\rightarrow \mathbb{Z}\overset{f'}\rightarrow \mathbb{Z}_2,$$
where the map $f'$ is surjective. Thus we can conclude that $[\Sigma^4\mathbb{C}P^{n-1},\Omega W_n]$ is isomorphic to $\mathbb{Z}\oplus \mathbb{Z}\oplus\mathbb{Z}_2$ and localized at odd prime $p$, is a free abelian group isomorphic to $\mathbb{Z}\oplus \mathbb{Z}$. Thus the homomorphism $\lambda$ is monomorphism.
\end{proof}
Similar to the Lemma \ref{b3}, we get the following lemma.
\begin{lem}\label{b10}
Im $\lambda\circ (\Omega\pi)_*$ is generated by\\
$i\colon \alpha_1=(3!, \frac{1}{2}4!), \alpha_2=(0, 4!)\quad \text{if}\quad b=2$,\\
$ii\colon \alpha_1=(3, 5), \alpha_2=(3, 5\cdot 3), \alpha_3=(0, 5\cdot 3)\quad \text{if}\quad b=3$, localized at odd prime $p$. $\quad\Box$
\end{lem}
Again, localized at odd prime $p$, if $b=2$ then $\alpha_1+\alpha_2=3(1, 1)$ and if $b=3$ then $5(3\alpha_1-\alpha_2)+\alpha_3= 5 \cdot 3(1, 1)$, therefore we get the following proposition.
\begin{prop}\label{b11}
Localized at odd prime $p$, the order of the Samelson product $\langle \varepsilon_{1,b+1}, \varepsilon_{b,b+1}\rangle$ is $3$, if $b=2$ and is $5\cdot 3$, if $b=3$ . $\quad\Box$
\end{prop}
\textbf{Proof of Theorem \ref{a1}}\\
By Propositions \ref{b5}, \ref{b8} and \ref{b11}, we get the Theorem \ref{a1}. $\quad\Box$
\section{The Samelson product $\langle \varepsilon_{a,n}, \varepsilon_{b,n}\rangle$ when $a+b = n+k$}\label{c0}
All spaces and maps are to be localized at odd prime $p$ throughout this section. In this section we study the order of the Samelson product $\langle \varepsilon_{a,n}, \varepsilon_{b,n}\rangle$ when $a+b=n+k$ for $k=1, 2, 3$ and will calculate the order of the Samelson product $\langle \varepsilon_{a,n}, \varepsilon_{b,n}\rangle$ at any odd primes in case $k=1$ for $a=2, b=2,3,4$, $a=b=3$, in case $k=2$ for $a=b=3$ and in case $k=3$ for $a=b=4$.\par
Put $Z=\Sigma\mathbb{C}P^a\wedge \Sigma\mathbb{C}P^b$. We have the following exact sequence
$$\tilde{K}^0(Z)\overset{(\Omega \pi')_*}\longrightarrow [Z,\Omega W_n]\overset{{\delta'}_*}\longrightarrow U_n(Z)\rightarrow 0.$$
Note that for $n\leq p(p-1)-1$, By [13, Corollary 3.2 ], we have the following lemma.
\begin{lem}\label{c1}
Localized at odd prime $p$, the group $[\Sigma\mathbb{C}P^a\wedge \Sigma\mathbb{C}P^b, \Omega W_n]$ is a free abelian group, where $a+b=n+k$ and $0\leq k \leq n-1$. $\quad\Box$
\end{lem}
Here, we define a homomorphism
$$\lambda'\colon [Z,\Omega W_n]\rightarrow H^{2n}(Z)\oplus H^{2n+2}(Z)\oplus\cdots\oplus H^{2n+2k+2}(Z)\cong \mathbb{Z}^{\oplus k+2}, $$
by $\lambda'(\alpha)=(\alpha^*(a_{2n}), \alpha^*(a_{2n+2}),\cdots,\alpha^*(a_{2n+2k+2}))$, where $\alpha\in [Z,\Omega W_n]$ and $a_{2n}$, $a_{2n+2}$, $\cdots$, $a_{2n+2k+2}$ are generators of $H^{2n}(\Omega W_n)\cong H^{2n+2}(\Omega W_n)\cong \cdots \cong H^{2n+2k+2}(\Omega W_n)\cong\mathbb{Z}$, respectively. Also, for $n\leq p(p-1)-1$, by [13, Lemma 3.3 ], we get the following lemma.
\begin{lem}\label{c2}
Localized at odd prime $p$, the map of $\lambda'\colon [Z,\Omega W_n]\rightarrow \mathbb{Z}^{\oplus k+2}$ is monic. $\quad\Box$
\end{lem}
\subsection{The case $k=1$}
$\bullet\quad a=b=2$\\
According to the definition of the map $\lambda'$, we have
\begin{align*}
\lambda' \circ(\Omega \pi')_* (\zeta_1\otimes(x\otimes x))=((\Omega \pi'\circ (\zeta_1\otimes(x\otimes x)))^*(a_6), (\Omega \pi'\circ &(\zeta_1\otimes(x\otimes x)))^*(a_8),\\& (\Omega \pi'\circ (\zeta_1\otimes(x\otimes x)))^*(a_{10}) ).
\end{align*}
Similar to the proof of Lemma \ref{b3}, the calculation of the components is as follows
\begin{align*}
&(\Omega \pi'\circ (\zeta_1\otimes(x\otimes x)))^*(a_6)=3!\sigma^2(t\otimes t),\\
&(\Omega \pi'\circ (\zeta_1\otimes(x\otimes x)))^*(a_8)=4!\sigma^2(\frac{1}{2}t^2\otimes t+t\otimes \frac{1}{2}t^2),\\
&(\Omega \pi'\circ (\zeta_1\otimes(x\otimes x)))^*(a_{10})=5!\sigma^2(\frac{1}{2}t^2\otimes \frac{1}{2}t^2).
\end{align*}
Therefore we obtain $\alpha_{1,1}=\lambda' \circ(\Omega \pi')_* (\zeta_1\otimes(x\otimes x))=(3!,\frac{1}{2}4!, \frac{1}{2}4!, \frac{1}{4}5!)$. Similarly, we can show that $\alpha_{1,2}=(0,4!,0, \frac{1}{2}5!)$, $\alpha_{2,1}=(0, 0, 4!, \frac{1}{2}5!)$ and $\alpha_{2,2}=(0,0,0,5!)$. Therefore we get the following lemma.
\begin{lem}\label{c3}
Im $\lambda'\circ (\Omega\pi')_*$ is generated by $\alpha_{1,1}$, $\alpha_{1,2}$, $\alpha_{2,1}$ and $\alpha_{2,2}$. $\quad\Box$
\end{lem}
By Lemma \ref{c3}, we know Im $\lambda'\circ (\Omega\pi')_*$ is generated by $\alpha_{1,1}$, $\alpha_{1,2}$, $\alpha_{2,1}$ and $\alpha_{2,2}$. On the other hand, it follows from Lemma \ref{b4} that $\lambda'(\tilde{\gamma}\circ (\varepsilon_{2,3}\wedge \varepsilon_{2,3}))=(1,1,1,1)$. We observe that $5(\alpha_{1,1}-\alpha_{1,2})+5(\alpha_{1,1}-\alpha_{2,1})+\alpha_{2,2} = 5\cdot 3(1,1, 1, 1)$. Thus we get the following proposition.
\begin{prop}\label{c4}
$|[\tilde{\gamma}\circ (\varepsilon_{2,3}\wedge \varepsilon_{2,3})]|= 5\cdot 3$. $\quad\Box$
\end{prop}
$\bullet\quad a=2, b=3$\\
Similar to the Lemma \ref{c3}, we get the following lemma.
\begin{lem}\label{c5}
Im $\lambda'\circ (\Omega\pi')_*$ is generated by $\alpha_{1,1}, \cdots, \alpha_{2,3}$, where
\begin{align*}
&\alpha_{1,1}=(\frac{1}{2}4!,\frac{1}{3!}5!,\frac{1}{2}4!,\frac{1}{4}5!,\frac{1}{2\cdot 3!}6!), && \alpha_{1,2}=(4!,\frac{1}{2}5!,0,\frac{1}{2}5!,\frac{1}{4}6!), &&&\alpha_{1,3}=(0,2\cdot 5!,0,0,6!), \\
&\alpha_{2,1}=(0,0,4!,\frac{1}{2}5!, \frac{1}{3!}6!), && \alpha_{2,2}=(0,0,0,5!,\frac{1}{2}6!), &&&\alpha_{2,3}=(0,0,0,0,2\cdot 6!). \quad\Box
\end{align*}
\end{lem}
On the other hand, it follows from Lemma \ref{b4} that $\lambda'(\tilde{\gamma}\circ (\varepsilon_{2,4}\wedge \varepsilon_{3,4}))=(1,1,1,1,1)$. Arguing as for Proposition \ref{c4}, we obtain the following proposition.
\begin{prop}\label{c6}
$|[\tilde{\gamma}\circ (\varepsilon_{2,4}\wedge \varepsilon_{3,4})]|= 5\cdot 3^2$. $\quad\Box$
\end{prop}
Similarly we get the following proposition.
\begin{prop}\label{c7}
The following hold:
\begin{itemize}
  \item if $a=2$ and $b=4$, then $|[\tilde{\gamma}\circ (\varepsilon_{2,5}\wedge\varepsilon_{4,5})]|=7\cdot 5^2\cdot 3^2$,
  \item if $a=b=3$, then $|[\tilde{\gamma}\circ (\varepsilon_{3,5}\wedge\varepsilon_{3,5})]|=7\cdot 5\cdot 3^2$. $\quad\Box$
\end{itemize}
\end{prop}
\subsection{The cases $k=2$ and $3$}
In the following, in cases $k=2, a=b=3$ and $k=3, a=b=4$ we will calculate the order of the Samelson product $\langle \varepsilon_{a,n}, \varepsilon_{b,n}\rangle$ when $a+b=n+k$, respectively.\par
Again, similar to Lemma \ref{c3}, we have the following lemma.
\begin{lem}\label{c8}
In case $k=2, a=b=3$, Im $\lambda'\circ (\Omega\pi')_*$ is generated by $\alpha_{1,1}, \cdots, \alpha_{3,3}$ and in
case $k=3, a=b=4$ is generated by $\beta_{1,1}, \cdots, \beta_{4,4}$, where
\begin{align*}
&\alpha_{1,1}=(\frac{1}{2}4!,\frac{1}{3!}5!,\frac{1}{2}4!,\frac{1}{4}5!,\frac{1}{2\cdot 3!}6!, \frac{1}{3!}5!,\frac{1}{2\cdot 3!}6!,\frac{1}{3!\cdot 3!}7!),&&\\ &\alpha_{1,2}=(4!,\frac{1}{2}5!,0,\frac{1}{2}5!,\frac{1}{4}6!,0,\frac{1}{3!}6!,\frac{1}{2\cdot 3!}7!),&&
\alpha_{1,3}=(0,2\cdot 5!,0,0,6!,0,0,\frac{1}{3}7!),\\
&\alpha_{2,1}=(0,4!,\frac{1}{2}5!,0,\frac{1}{3!}6!,\frac{1}{2}5!,\frac{1}{4}6!,\frac{1}{2\cdot 3!}7!), && \alpha_{3,1}=(0,0,0,0,0,2\cdot 5!,6!, \frac{1}{3}7!),\\
& \alpha_{2,2}=(0,0,0,5!,\frac{1}{2}6!,0,\frac{1}{2}6!,\frac{1}{4}7!), && \alpha_{3,2}=(0,0,0,0,0,0,2\cdot 6!,7!),\\
&\alpha_{2,3}=(0,0,0,0,2\cdot 6!,0,0,7!), &&\alpha_{3,3}=(0,0,0,0,0,0,0,4\cdot 7!).
\end{align*}
and also
\begin{align*}
&\beta_{1,1}=(\frac{1}{3!}5!,\frac{1}{4!}6!,\frac{1}{4}5!,\frac{1}{2\cdot3!}6!,\frac{1}{2\cdot 4!}7!, \frac{1}{3!}5!,\frac{1}{2\cdot 3!}6!,\frac{1}{3!\cdot 3!}7!,\frac{1}{3!\cdot 4!}8!,\frac{1}{4!}6!,\frac{1}{2\cdot 4!}7!,\frac{1}{3!\cdot 4!}8!,\frac{1}{4!\cdot 4!}9! ),\\ &\beta_{1,2}=(\frac{1}{2}5!,\frac{5}{12}6!,\frac{1}{2}5!,\frac{1}{4}6!,\frac{5}{24}7!, 0,\frac{1}{3!}6!,\frac{1}{2\cdot 3!}7!,\frac{5}{3!\cdot 12}8!,0,\frac{1}{4!}7!,\frac{1}{2\cdot4!}8!,\frac{5}{4!\cdot 12}9! ),\\
&\beta_{1,3}=(2\cdot 5!,\frac{3}{2}6!,0,6!,\frac{3}{4}7!,0,0,\frac{2}{3!}7!,\frac{3}{2\cdot3!}8!,0,0,\frac{2}{4!}8!, \frac{3}{2\cdot 4!}9! ), \\
&\beta_{1,4}=(0,5\cdot 6!,0,0,\frac{5}{2}7!,0,0,0,\frac{5}{3!}8!,0,0,0,\frac{5}{4!}9!),\\
&\beta_{2,1}=(0,0,\frac{1}{2}5!,\frac{1}{3!}6!,\frac{1}{4!}7!,\frac{1}{2}5!,\frac{1}{4}6!,\frac{1}{2\cdot3!}7!,
\frac{1}{2\cdot4!}8!,\frac{5}{12}6!,\frac{5}{24}7!,\frac{5}{3!\cdot 12}8!,\frac{5}{4!\cdot 12}9!),\\
&\beta_{2,2}=(0,0,5!,\frac{1}{2}6!,\frac{5}{12}7!,0,\frac{1}{2}6!,\frac{1}{4}7!,\frac{5}{24}8!,0,
\frac{5}{12}7!,\frac{5}{24}8!,\frac{25}{12\cdot12}9!),\\
&\beta_{2,3}=(0,0,0,2\cdot 6!,\frac{3}{2}7!,0,0,7!,\frac{3}{4}8!,0,0,\frac{10}{12}8!,\frac{15}{24}9!),\\
&\beta_{2,4}=(0,0,0,0,5\cdot 7!,0,0,0,\frac{5}{2}8!,0,0,0,\frac{25}{12}9!),\\
&\beta_{3,1}=(0,0,0,0,0,2\cdot 5!,6!, \frac{2}{3!}7!, \frac{2}{4!}8!,\frac{3}{2}6!,\frac{3}{4}7!,\frac{3}{2\cdot3!}8!,\frac{3}{2\cdot4!}9!),\\
& \beta_{3,2}=(0,0,0,0,0,0,0,2\cdot 6!,7!,\frac{5}{6}8!,0,\frac{3}{2}7!,\frac{3}{4}8!,\frac{15}{24}9!),
\end{align*}
\begin{align*}
&\beta_{3,3}=(0,0,0,0,0,0,0,4\cdot 7!,3\cdot 8!,0,0,3\cdot 8!,\frac{9}{4}9!),
&&\beta_{3,4}=(0,0,0,0,0,0,0,0,10\cdot 8!,0,0,0,\frac{15}{2}9!),\\
&\beta_{4,1}=(0,0,0,0,0,0,0,0,0,5\cdot 6!,\frac{5}{2}7!,\frac{5}{3!}8!,\frac{5}{4!}9!),
&& \beta_{4,2}=(0,0,0,0,0,0,0,0,0,0,5\cdot 7!,\frac{5}{2}8!,\frac{25}{12}9!),\\
&\beta_{4,3}=(0,0,0,0,0,0,0,0,0,0,0,10\cdot 8!,\frac{15}{2}9!),
&&\beta_{4,4}=(0,0,0,0,0,0,0,0,0,0,0,0,25\cdot 9!). \quad\Box
\end{align*}
\end{lem}
On the other hand, it follows from Lemma \ref{b4} that
$$\lambda'(\tilde{\gamma}\circ (\varepsilon_{3,4}\wedge\varepsilon_{3,4}))=(1,1,1,1,1,1,1,1),$$
$$\lambda'(\tilde{\gamma}\circ (\varepsilon_{4,5}\wedge\varepsilon_{4,5}))=(1,1,1,1,1,1,1,1,1,1,1,1,1).$$
Arguing as for Proposition \ref{c4}, we obtain the following proposition.
\begin{prop}\label{c9}
The following hold:
\begin{itemize}
  \item if $k=2$ and $a=b=3$, then $|[\tilde{\gamma}\circ (\varepsilon_{3,4}\wedge\varepsilon_{3,4})]|=7\cdot 5\cdot 3^2$,
  \item if $k=3$ and $a=b=4$, then $|[\tilde{\gamma}\circ (\varepsilon_{4,5}\wedge\varepsilon_{4,5})]|=7\cdot 5^3\cdot 3^4$. $\quad\Box$
\end{itemize}
\end{prop}
\textbf{Proof of Theorem \ref{a2}}\\
By Propositions \ref{c4}, \ref{c6}, \ref{c7} and \ref{c9}, we get the Theorem \ref{a2}. $\quad\Box$
\section{Applications}\label{d0}
In this section, we describe two applications of the Samelson products in $SU(n)$. One: The homotopy commutativity of $SU(n)$ localized at a prime $p$ for $n=3, 4$ and $5$. Two: The order of the commutator on $SU(4)$ localized at an odd prime $p$.
\subsection{The homotopy commutativity of $SU(n)$}
In \cite{[M]} McGibbon shows that if $p > 2n$ then $SU(n)_{(p)}$ is homotopy commutative. Here for $n=3, 4$ and $5$ we give a new proof of this fact.\\
\textbf{Proof of Corollary \ref{a3}}\\
Let $n=3$. For CW-complexes $X$ and $Y$ , we denote the adjoint isomorphism between the homotopy sets by
$$ad\colon [\Sigma X, Y ]\overset{iso}\longrightarrow [X, \Omega Y ]. $$
Put $A = \Sigma\mathbb{C}P^2$. Localized at $p > 6$, by Proposition \ref{c4} the Samelson product $\langle \varepsilon_2,\varepsilon_2 \rangle\colon A\wedge A\rightarrow SU(3)$ is equal to zero. Therefore taking the adjoint, the generalized Whitehead product $[\phi_2, \phi_2]\colon \Sigma (A \wedge A) \rightarrow BSU(3)$ vanish, where $\phi_2= {ad}^{-1}\varepsilon_2$. Consider the inclusion map $\Sigma A \hookrightarrow \Sigma SU(3)$, there is a retraction $r\colon \Sigma SU(3)\rightarrow\Sigma A$ such that the composition $\Sigma A \rightarrow \Sigma SU(3)\rightarrow \Sigma A$ is the identity of $\Sigma A$. We can choose the retraction $r$ such that the composition $\Sigma SU(3) \overset{r}\rightarrow \Sigma A \overset{\phi_2}\rightarrow BSU(3)$ is equal to ${ad}^{-1}1_{SU(3)}$. Therefore the generalized Whitehead product $[{ad}^{-1}1_{SU(3)}, {ad}^{-1}1_{SU(3)}] \colon \Sigma(SU(3) \wedge SU(3)) \rightarrow BSU(3)$ also vanishes and thus the proof is completed. For $n=4$ and $5$, by the Proposition \ref{c9} we have the similar proof. $\quad\Box$
\subsection{The order of the commutator on $SU(4)$}
A. Kono and S. Theriault in \cite{[KT]} showed that the order of the commutator on $SU(3)$ is equal to $2^3\cdot 3 \cdot 5$. They also gave an upper bound on the number of homotopy types of gauge groups for principal $SU(3)$-bundle over an $n$-sphere. Here we will study this subject on $SU(4)$ localized at an odd prime $p$. The calculation of the $2$-component of the order of $SU(4)\wedge SU(4) \rightarrow SU(4)$ is very hard, so we decline to calculate it. We localize all spaces and maps at an odd prime $p$. For a prime $p$ and an integer $m$, let $\nu_p(m)$ be the largest
integer $r$ such that $p^r$ divides $m$. The following theorem was proved in \cite{[KT]}.
\begin{thm}\label{d1}
Let $G$ be a compact, connected Lie group and let $Y$ be a space. Fix a homotopy
class $[f]\in [\Sigma Y, BG]$. For any integer $k$, let $P_k \rightarrow \Sigma Y$ be the principal $G$-bundle induced by $kf$, and let $\mathcal{G}_k$ be its gauge group. If the order of the commutator $G\wedge G \overset{c}\rightarrow G$ is $m$, then the number of distinct $p$-local homotopy types for the gauge groups $\mathcal{G}_k$ is at most $\nu_p(m)+1$. $\quad\Box$
\end{thm}
Let $X$ be an $H$-space with multiplicative map $\mu \colon X \times X \rightarrow X$. The composite
$$\mu^* \colon \Sigma X\wedge X\rightarrow \Sigma (X\times X)\overset{\Sigma \mu} \longrightarrow \Sigma X$$
is known as the Hopf construction on $X$ (for details see \cite{[T]}). In particular, if $G$ is a simply-connected, simple compact Lie group then there is a Hopf construction $\mu^* \colon \Sigma G\wedge G\rightarrow \Sigma G$. The adjoint of the identity map on $SU(4)$ is the evaluation map $ev \colon \Sigma SU(4) \rightarrow BSU(4)$. Let $f$ be the composite $f \colon \Sigma^2\mathbb{C}P^3\overset{\Sigma \varepsilon_{3,4}}\longrightarrow  \Sigma SU(4) \overset{ev}\longrightarrow  BSU(4)$.
\begin{lem}\label{d2}
Localized at an odd prime $p$, the following diagram is homotopy commutative.
\begin{equation}
\begin{tikzpicture}[baseline=(current bounding box.center)]
\matrix(m)[matrix of math nodes,
 row sep=2em, column sep=2em,
 text height=1.5ex, text depth=0.20ex]
   {\Sigma SU(4)&BSU(4) \\ \Sigma^2\mathbb{C}P^3& BSU(4)\\}; \path[->,font=\scriptsize]
  (m-1-1) edge node[above] {$ev$}(m-1-2)
  (m-1-1) edge node[left] {$\theta$} (m-2-1)
  (m-1-2) edge node[right] {$||$}(m-2-2)
    (m-2-1) edge node[above] {$f$} (m-2-2);
 \end{tikzpicture}
 \end{equation}
where $\theta$ is a left homotopy inverse for $\Sigma \varepsilon_{3,4}$.
\end{lem}
\begin{proof}
By \cite{[MNT]}, there is a homotopy equivalence $SU(4)\simeq B(3, 7)\times S^5$, where $H^*(B(3, 7))\cong
\Lambda(x_3, x_7)$ and also there is a homotopy fibration $S^3\rightarrow B(3, 7)\rightarrow S^7$. Thus we have the
homotopy equivalence
$$\Sigma SU(4) \simeq \Sigma B(3, 7)\vee S^6 \vee \Sigma^6 B(3, 7). \quad(\star\star)$$
By \cite{[JW]}, we know that $B(3, 7)$ has a cell-structure $B(3, 7)\simeq S^3 \cup_{\nu} e^7 \cup e^{10}$, where $\nu\in \pi_6(S^3)$ is the attaching map. By [3, Lemma 4.1 ], we have that the top cell splits off after a suspension. This then gives us $\Sigma B(3, 7)\simeq (S^4 \cup e^8)\vee S^{11}$ and $\Sigma^6 B(3, 7)\simeq (S^9 \cup e^{13})\vee S^{16}$. Thus we can regard the evaluation map $ev$ as the wedge sum $\Sigma \varepsilon_{3,4} +g$ for some map $g\colon S^{11}\vee \Sigma^6 B(3, 7) \rightarrow BSU(4)$. Therefore we need to show $g \simeq \ast$. By \cite{[G]}, the homotopy fibre of the evaluation map $ev\colon \Sigma SU(4)\rightarrow BSU(4)$ is the Hopf construction $\mu^*\colon \Sigma SU(4)\wedge SU(4)\rightarrow  \Sigma SU(4)$. Note that $H_*(SU(4))=\Lambda(x, y, z)$, where $x, y, z$ have degrees $3, 5, 7$ respectively, and in mod-$p$ homology the restriction of $\mu^*$ to
$$S^{11}\vee \Sigma^6 B(3, 7) \simeq (\Sigma S^3 \wedge S^5)\vee (\Sigma S^3 \wedge S^7) \vee(\Sigma S^5 \wedge S^7) \vee(\Sigma S^3 \wedge S^5 \wedge S^7)$$
is an isomorphism onto the submodule of $\Sigma \Lambda (x, y, z)$ generated by the suspensions of $xy, xz,
yz$ and $xyz$. Thus we can choose the equivalence $(\star\star)$ so that the $S^{11}\vee \Sigma^6 B(3, 7)$ summand
of $\Sigma SU(4)$ factors through $\mu^*$. Doing so, we obtain $g \simeq \ast$. Therefore the evaluation map $ev$ factors through $f$ and the lemma follows.
\end{proof}
We will prove the following lemma that plays an important role in computing the order of the commutator on $SU(4)$.
\begin{lem}\label{d3}
Localized at an odd prime $p$, the order of the maps $SU(4)\wedge SU(4)\overset{c}\rightarrow SU(4)$
and $\Sigma\mathbb{C}P^3 \wedge \Sigma\mathbb{C}P^3 \overset{\varepsilon_{3,4}\wedge \varepsilon_{3,4}}\longrightarrow SU(4)\wedge SU(4)\overset{c}\rightarrow SU(4)$ are equal.
\end{lem}
\begin{proof}
Let the composite $c \circ (\varepsilon_{3,4}\wedge \varepsilon_{3,4})$ have order $t$. By taking adjoint, the composite
$$[f, f]\colon \Sigma^2\mathbb{C}P^3 \wedge \Sigma\mathbb{C}P^3 \overset{\Sigma\varepsilon_{3,4}\wedge \varepsilon_{3,4}}\longrightarrow \Sigma SU(4)\wedge SU(4)\overset{[ev,ev]}\longrightarrow BSU(4)$$
has order $t$, where the Whitehead product $[ev, ev]$ is the adjoint of the commutator map
$SU(4)\wedge SU(4)\overset{c}\rightarrow SU(4)$. On the other hand, $[f, f]$ is homotopic to the composite
$$ \Sigma^2\mathbb{C}P^3 \wedge \Sigma\mathbb{C}P^3 \rightarrow  \Sigma^2\mathbb{C}P^3 \vee \Sigma^2\mathbb{C}P^3\overset{f\vee f}\longrightarrow BSU(4).$$
Since $[f, f]$ has order $t$, $t\cdot(f\vee f)$ extends to a map $\mu\colon \Sigma^2\mathbb{C}P^3 \times \Sigma^2\mathbb{C}P^3 \rightarrow BSU(4)$. Now by Lemma \ref{d2} there is a map $\bar{\mu}\colon \Sigma SU(4)\times \Sigma SU(4)\rightarrow BSU(4)$ such that the restriction of $\bar{\mu}$ on $\Sigma SU(4)\vee\Sigma SU(4)$ is $t\cdot(ev\vee ev)$. Thus the composite
$$\Sigma SU(4)\wedge SU(4) \rightarrow \Sigma SU(4)\vee\Sigma SU(4)\overset{t\cdot(ev\vee ev)}\longrightarrow  BSU(4)$$
is null homotopic. We know that this composite is homotopic to $t\cdot [ev, ev]$. Thus $[ev, ev]$ has order $t$, which implies that its adjoint $c$ also has order $t$. To converse direction, since the composite $c \circ (\varepsilon_{3,4}\wedge \varepsilon_{3,4})$ factors through $c$, therefore the order of $c \circ (\varepsilon_{3,4}\wedge \varepsilon_{3,4})$ less than the order of $c$ and proof is complete.
\end{proof}
\textbf{Proof of Theorem \ref{a4}}\\
By Proposition \ref{c9} and Lemma \ref{d3}, we can conclude the Theorem \ref{a4}. $\quad\Box$\\\\
Now by Theorems \ref{a4} and \ref{d1}, we get the following corollary.
\begin{cor}\label{d4}
For any homotopy class $[f] \in [\Sigma Y, BSU(4)]$ the number of distinct $p$-local homotopy types for the gauge groups $\mathcal{G}_k$ is at most $3$ if $p=3$, $2$ if $p=5$ or $p=7$, and $1$ if $p > 7$. $\quad\Box$
\end{cor}
\section{The group $[\Sigma^{4m-1}Q_{n-m+1}, Sp(n)]$}\label{e0}
Our main goal in this section is to study the Samelson product $\langle \varepsilon_{m,n},\epsilon_{m,n} \rangle$, where the map $\epsilon_{m,n}\colon Q_{n-m+1}\rightarrow Sp(n)$ is the inclusion map. We denote $Sp(\infty)/Sp(n)$ by $X_n$ and $[X,Sp(n)]$ by $Sp_n(X)$. Let $Q_n$ be the symplectic quasi projective space of rank $n$ defined in \cite{[J]}. This space has the following cellular structure
$$Q_n=S^3\cup e^7\cup e^{11}\cup \cdots \cup e^{4n-1}.$$
Put $X=S^{4m-1}\wedge Q_{n-m+1}$. Note that $X$ has the following cellular structure
$$X\simeq S^{4m+2}\cup e^{4m+6}\cup e^{4m+10}\cup \cdots \cup e^{4n+2}.$$
Recall that as an algebra $H^*(Sp(n);\mathbb{Z})= \bigwedge (y_3,y_7,\cdots, y_{4n-1})$, $H^*(Sp(\infty);\mathbb{Z})= \bigwedge (y_3,y_7,\cdots)$ and $H^*(BSp(\infty);\mathbb{Z})= \mathbb{Z}[q_1,q_2,\cdots]$, where $y_{4i-1}=\sigma q_i$, $\sigma$ is the cohomology suspension and $q_i$ is the $i-$th symplectic Pontrjagin class. Consider the projection map $\pi:Sp(\infty)\rightarrow X_n$, also an algebra we have
\begin{align*}
& H^*(X_n;\mathbb{Z})= \bigwedge (\bar{y}_{4n+3}, \bar{y}_{4n+7}, \cdots), \\
& H^*(\Omega X_n;\mathbb{Z})= \mathbb{Z}\{b_{4n+2}, b_{4n+6}, \cdots, b_{8n+2} \} \quad (*\leq 8n+2 ),
\end{align*}
where $\pi^*(\bar{y}_{4i+3})=y_{4i+3}$ and $b_{4n+4j-2}=\sigma (\bar{y}_{4n+4j-1})$. Consider the following fibre sequence
\begin{eqnarray}
\Omega Sp(\infty)\overset{\Omega \pi}\longrightarrow \Omega X_n\overset{\delta}\longrightarrow Sp(n) \overset{j}\longrightarrow Sp(\infty)\overset{\pi}\longrightarrow X_n.
\end{eqnarray}
By applying the functor $[X,\quad]$ to fibration $(5.1)$, we get the following exact sequence
$$
[X,\Omega Sp(\infty)]\overset{(\Omega \pi)_*}\longrightarrow [X,\Omega X_n]\overset{\delta_*}\longrightarrow Sp_n(X) \overset{j_*}\longrightarrow [X,Sp(\infty)]\overset{\pi_*}\longrightarrow [X,X_n].\quad (\star)
$$
Note that $X_n$ has the cellular structure $X_n\simeq S^{4n+3}\cup e^{4n+7}\cup e^{4n+11}\cup \cdots$, also we have $\Omega X_n\simeq S^{4n+2}\cup e^{4n+6}\cup e^{4n+10}\cup \cdots$. According to the $CW$-structure of $X_n$, we have the following isomorphisms
\begin{equation*}
\pi_i(X_n)=0 \quad (for \hspace{2mm} i\leq 4n+2), \qquad \pi_{4n+3}(X_n)\cong \mathbb{Z}.
\end{equation*}
Observe that $[X, Sp(\infty)]\cong [\Sigma X, BSp(\infty)]\cong {\widetilde{KSp}}^{-1}(X)$. Apply ${\widetilde{KSp}}^{-1}$ to the cofibration sequence $S^{4m-1}\wedge Q_{n-m}\rightarrow X\rightarrow S^{4n+2}$. We know that ${\widetilde{KSp}}^{-1}(S^{4n+2})=0$, for every $n\geq 1$, so by use of induction, we can conclude that ${\widetilde{KSp}}^{-1}(X)=0$. On the other hand, we know that $\Omega X_n$ is $(4n+1)$-connected and $H^{4n+2}(\Omega X_n)\cong \mathbb{Z}$ which is generated by $b_{4n+2}=\sigma(\bar{y}_{4n+3})$. The map $b_{4n+2}\colon\Omega X_n \rightarrow K(\mathbb{Z},4n+2)$ is a loop map and is a $(4n+3)$-equivalence. Since dim$X\leq 4n+2$, the map $(b_{4n+2})_*\colon[X,\Omega X_n] \rightarrow H^{4n+2}(X)$ is an isomorphism of groups. Thus we rewrite the exact sequence $(\star)$ as the following exact sequence
\begin{eqnarray*}
{\widetilde{KSp}}^{-2}(X)\overset{\psi}\longrightarrow H^{4n+2}(X)\rightarrow Sp_n(X)\rightarrow 0,
\end{eqnarray*}
where we use the isomorphism ${\widetilde{KSp}}^{-i}(X)\cong [\Sigma^iX,BSp(\infty)]$. So we have the exact sequence
\begin{equation*}
0\rightarrow Coker \psi \overset{\iota}\longrightarrow Sp_n(X)\rightarrow 0.
\end{equation*}
Therefore we get the following lemma.
\begin{lem}\label{e1}
$Sp_n(X)\cong Coker \psi$. $\quad \Box$
\end{lem}
In what follows, we will calculate the image of $\psi$.\par
Let $Y$ be a $CW$-complex with dim $Y \leq 4n+ 2$, we denote $[Y, U(2n+1)]$ by $U_{2n+1}(Y)$. By [7, Theorem 1.1], there is an exact sequence
\begin{equation*}
\tilde{K}^{-2}(Y)\overset{\varphi}\longrightarrow H^{4n+2}(Y)\rightarrow U_{2n+1}(Y)\rightarrow \tilde{K}^{-1}(Y)\rightarrow 0,
\end{equation*}
for any $f\in \tilde{K}^{-2}(Y)$ the map $\varphi$ is defined as follows
$$\varphi(f)=(2n+1)!ch_{2n+1}(f),$$
where $ch_i$ denotes the $2i$-dimensional part of the Chern character. Also, we use the isomorphism $\tilde{K}^{-i}(Y)\cong [\Sigma^iY,BU(\infty)]$. In this paper, we denote both the canonical inclusion $Sp(n)\hookrightarrow U(2n)$ and the induced
map ${\widetilde{KSp}}^{*}(X) \rightarrow {\tilde{K}}^{*}(X)$ by $c'$. By [18, Theorem 1.3], there is a commutative diagram
\begin{equation}
\begin{tikzpicture}[baseline=(current bounding box.center)]
\matrix(m)[matrix of math nodes,
 row sep=2em, column sep=2em,
 text height=1.5ex, text depth=0.20ex]
   {{\widetilde{KSp}}^{-2}(X)& H^{4n+2}(X)\\ {\tilde{K}}^{-2}(X)& H^{4n+2}(X).\\}; \path[->,font=\scriptsize]
  (m-1-1) edge node[above] {$\psi$}(m-1-2)
  (m-1-1) edge node[left] {$c'$} (m-2-1)
  (m-1-2) edge node[right] {$(-1)^{n+1}$}(m-2-2)
      (m-2-1) edge node[above] {$\varphi$} (m-2-2);
 \end{tikzpicture}
 \end{equation}
Therefore to calculate the image of $\psi$, we first calculate the image of $\varphi$. We denote the free abelian group with a basis $e_1,e_2,\cdots$, by $\mathbb{Z}\{ e_1,e_2,\cdots \}$ and the direct sum of $k$ copies of $\mathbb{Z}$ by $\mathbb{Z}^k$. We need the following lemmas.
\begin{lem}\label{e2}
The group ${\widetilde{KSp}}^{-2}(\Sigma^{4m}Q_{n-m})$ is isomorphic to a finite cyclic group.
\end{lem}
\begin{proof}
To prove this lemma we will use induction. Note that there are isomorphisms
\begin{align*}
&{\widetilde{KSp}}^{-2}(S^{4i+2})\cong \mathbb{Z}, \quad &&{\widetilde{KSp}}^{-2}(S^{4n+1})\cong 0,\\ &{\widetilde{KSp}}^{-2}(S^{4n})\cong \left\{ \begin{array}{ll}
0 & \textrm{if $n$ is even },\\
\mathbb{Z}/2\mathbb{Z} & \textrm{if $n$ is odd},\\
\end{array} \right.\quad &&{\widetilde{KSp}}^{-2}(S^{4n-1})\cong \left\{ \begin{array}{ll}
\mathbb{Z}/2\mathbb{Z} & \textrm{if $n$ is even },\\
0 & \textrm{if $n$ is odd}.\\
\end{array} \right.
\end{align*}
First, for two cases $Q_2$ and $Q_3$,  we calculate the groups ${\widetilde{KSp}}^{-2}(\Sigma^{4m}Q_2)$ and ${\widetilde{KSp}}^{-2}(\Sigma^{4m}Q_3)$, respectively. Consider the homotopy cofibration sequence
$$S^{4m+6}\overset{\Sigma^{4m}v'}\longrightarrow S^{4m+3}\rightarrow \Sigma^{4m} Q_2\rightarrow S^{4m+7}\overset{\Sigma^{4m+1}v'}\longrightarrow S^{4m+4},$$
where $v'$ is a generator of $\pi_6(S^3)\cong \mathbb{Z}_{12}$. By applying ${\widetilde{KSp}}^{-2}$, we get the exact sequence
\begin{align*}
{\widetilde{KSp}}^{-2}(S^{4m+4})\overset{(\Sigma^{4m+1}v')^*}\longrightarrow {\widetilde{KSp}}^{-2}(S^{4m+7})\rightarrow {\widetilde{KSp}}^{-2}(\Sigma^{4m} Q_2)&\rightarrow {\widetilde{KSp}}^{-2}(S^{4m+3})\\&\rightarrow{\widetilde{KSp}}^{-2}(S^{4m+6}).
\end{align*}
Now, in cases where $m$ is even and $m$ is odd, we obtain the exact sequences
$$\mathbb{Z}_2\overset{(\Sigma^{4m+1}v')^*}\longrightarrow \mathbb{Z}_2 \rightarrow  {\widetilde{KSp}}^{-2}(\Sigma^{4m} Q_2) \rightarrow 0,\quad
0 \rightarrow  {\widetilde{KSp}}^{-2}(\Sigma^{4m} Q_2) \rightarrow \mathbb{Z}_2\rightarrow \mathbb{Z},$$
respectively. Since $\mathbb{Z}$ is torsion-free, any map $\mathbb{Z}_2\rightarrow \mathbb{Z}$ is zero map, thus in case $m$ is odd, by exactness we obtain the group ${\widetilde{KSp}}^{-2}(\Sigma^{4m} Q_2)$ is isomorphic to $\mathbb{Z}_2$. Let $m$ be even. Since $Sp(\infty)$ is homotopy euivalent to $\Omega^4O(\infty)$, we can determine the map ${\widetilde{KSp}}^{-2}(S^{4m+4})\overset{(\Sigma^{4m+1}v')^*}\longrightarrow {\widetilde{KSp}}^{-2}(S^{4m+7})$ by $(\Sigma^{4m+1}v')^*\colon \pi_{4m+9}(SO(\infty))\rightarrow \pi_{4m+12}(SO(\infty))$. Let $t_1$ be a generator of $\pi_{4m+9}(SO(\infty))\cong \mathbb{Z}_2$. Then $(\Sigma^{4m+1}v')^* (t_1)$ is the composite $S^{4m+12}\overset{\Sigma^{4m+6}v'} \longrightarrow S^{4m+9}\overset{t_1} \longrightarrow SO(\infty)$. Note that $t_1\circ \Sigma^{4m+6}v'\simeq t_1\circ 2v_{4m+9}\simeq 2t_1\circ v_{4m+9}$ is null homotopic, where by \cite{[T1]} we have $\Sigma^{4m+6}v'\simeq 2v_{4m+9}$, stably. Thus the map $(\Sigma^{4m+1}v')^*$ is the zero map. Therefore we can conclude that the group ${\widetilde{KSp}}^{-2}(\Sigma^{4m} Q_2)$ is also isomorphic to $\mathbb{Z}_2$.\par
Now, we study the case $ Q_3$. By applying ${\widetilde{KSp}}^{-2}$ to the homotopy cofibration sequence
$$S^{4m+10}\overset{\theta}\longrightarrow \Sigma^{4m}Q_2\rightarrow \Sigma^{4m}Q_3\rightarrow S^{4m+11}\overset{\Sigma\theta}\longrightarrow \Sigma^{4m+1}Q_2,$$
we get the following exact sequence
\begin{align*}
{\widetilde{KSp}}^{-2}(\Sigma^{4m+1}Q_2)\overset{(\Sigma\theta)^*}\longrightarrow {\widetilde{KSp}}^{-2}(S^{4m+11})\rightarrow {\widetilde{KSp}}^{-2}(\Sigma^{4m} Q_3)&\rightarrow {\widetilde{KSp}}^{-2}(\Sigma^{4m} Q_2)\\&\overset{\theta^*}\longrightarrow {\widetilde{KSp}}^{-2}(S^{4m+10}).
\end{align*}
Now, in cases where $m$ is even and $m$ is odd, we obtain the exact sequences
$$0\rightarrow {\widetilde{KSp}}^{-2}(\Sigma^{4m} Q_3)\rightarrow \mathbb{Z}_2\overset{\theta^*}\rightarrow \mathbb{Z},\quad
\mathbb{Z}_2 \overset{(\Sigma\theta)^*}\rightarrow \mathbb{Z}_2\rightarrow {\widetilde{KSp}}^{-2}(\Sigma^{4m} Q_3) \rightarrow \mathbb{Z}_2\overset{\theta^*}\rightarrow \mathbb{Z}, $$
respectively. When $m$ is even then by exactness we obtain the group ${\widetilde{KSp}}^{-2}(\Sigma^{4m} Q_3)$ is isomorphic to $\mathbb{Z}_2$. Let $m$ be odd. We show that the map ${\widetilde{KSp}}^{-2}(\Sigma^{4m+1} Q_2)\overset{(\Sigma\theta)^*}\longrightarrow {\widetilde{KSp}}^{-2}(S^{4m+11})$ is trivial. Let $t_1$ be a generator of ${\widetilde{KSp}}^{-2}(\Sigma^{4m+1} Q_2)\cong [\Sigma^{4m+1} Q_2,\Omega Sp(\infty)]$. Then we know that $(\Sigma\theta)^*(t_1)$ is the composite $ S^{4m+11}\overset{\Sigma\theta} \longrightarrow \Sigma^{4m+1} Q_2\overset{t_1} \longrightarrow \Omega Sp(\infty)$, where the map $t_1$ comes from the composition $\Sigma^{4m+1} Q_2\overset{q}\rightarrow S^{4m+8}\overset{s} \rightarrow \Omega Sp(\infty) $, where $q$ is the pinch map to the top cell and $s$ generates ${\widetilde{KSp}}^{-2}(S^{4m+8})\cong \mathbb{Z}_2$. Consider the following commutative diagram
\begin{equation*}
\begin{tikzpicture}[baseline=(current bounding box.center)]
\matrix(m)[matrix of math nodes,
 row sep=2em, column sep=2em,
 text height=1.5ex, text depth=0.20ex]
   {S^{4m+11}& \Sigma^{4m+1} Q_2& \Omega Sp(\infty)\\ & S^{4m+8}& \Omega Sp(\infty)\\}; \path[->,font=\scriptsize]
  (m-1-1) edge node[above] {$\Sigma\theta$}(m-1-2)
  (m-1-1) edge node[left] {$\Sigma^{4m+5}\nu'$} (m-2-2)
  (m-1-2) edge node[above] {$s\circ q$}(m-1-3)
  (m-1-2) edge node[right] {$q$}(m-2-2)
  (m-1-3) edge node[right] {$||$}(m-2-3)
      (m-2-2) edge node[above] {$s$} (m-2-3);
 \end{tikzpicture}
 \end{equation*}
where the square clearly commutes and since $\Sigma^{4m+5}\nu'$ generates $\pi_{4m+11}(S^{4m+8})$, the triangle is also commutes. We wish to show that the upper direction around the diagram is trivial. We can determine the map ${\widetilde{KSp}}^{-2}(S^{4m+8})\rightarrow {\widetilde{KSp}}^{-2}(S^{4m+11})$ by $ \pi_{4m+13}(SO(\infty))\rightarrow \pi_{4m+16}(SO(\infty))$ that sends the generator $s$ to the composite $S^{4m+16}\overset{\Sigma^{4m+10}v'} \longrightarrow S^{4m+13}\overset{s} \longrightarrow SO(\infty)$. Note that $s\circ \Sigma^{4m+10}v'\simeq s\circ 2v_{4m+13}\simeq 2s\circ v_{4m+13}$ is null homotopic. Thus the map $(\Sigma\theta)^*$ is the zero map. Therefore we can conclude that the group ${\widetilde{KSp}}^{-2}(\Sigma^{4m}Q_3)$ is also isomorphic to $\mathbb{Z}_2\oplus \mathbb{Z}_2$ or $\mathbb{Z}_4$.\par
Now, by the induction, we let the Lemma hold for the case $Q_{n-m-1}$. Consider the cofibration sequence
$$S^{4n-2}\overset{\theta}\longrightarrow \Sigma^{4m}Q_{n-m-1}\rightarrow \Sigma^{4m} Q_{n-m}\rightarrow S^{4n-1} \overset{\Sigma\theta}\longrightarrow \Sigma^{4m+1} Q_{n-m-1}.$$
This sequence induces an exact sequence
\begin{align*}
{\widetilde{KSp}}^{-2}(\Sigma^{4m+1} Q_{n-m-1})\overset{(\Sigma\theta)^*}\longrightarrow {\widetilde{KSp}}^{-2}(S^{4n-1})\rightarrow {\widetilde{KSp}}^{-2}(\Sigma^{4m} Q_{n-m})&\rightarrow {\widetilde{KSp}}^{-2}(\Sigma^{4m} Q_{n-m-1})\\& \overset{\theta^*}\rightarrow {\widetilde{KSp}}^{-2}(S^{4n-2}).
\end{align*}
Now, in cases where $n$ is even and $n$ is odd, we get the exact sequences
$$\overset{(\Sigma\theta)^*}\longrightarrow \mathbb{Z}_2\rightarrow {\widetilde{KSp}}^{-2}(\Sigma^{4m} Q_{n-m})\rightarrow {\widetilde{KSp}}^{-2}(\Sigma^{4m} Q_{n-m-1}) \overset{\theta^*}\rightarrow \mathbb{Z},$$
$$0\rightarrow {\widetilde{KSp}}^{-2}(\Sigma^{4m} Q_{n-m})\rightarrow {\widetilde{KSp}}^{-2}(\Sigma^{4m} Q_{n-m-1}) \overset{\theta^*}\rightarrow \mathbb{Z},$$
where by the assumption of induction, the group ${\widetilde{KSp}}^{-2}(\Sigma^{4m} Q_{n-m-1})$ is isomorphic to a finite cyclic group. Therefore in both cases by exactness, we can conclude that the group ${\widetilde{KSp}}^{-2}(\Sigma^{4m} Q_{n-m})$ is also isomorphic to a finite cyclic group.
\end{proof}
\begin{lem}\label{e3}
The following hold:\\
$(a)$: there is an isomorphism ${\widetilde{KSp}}^{-2}(X)\rightarrow {\widetilde{KSp}}^{-2}(S^{4m+2})\oplus \cdots \oplus  {\widetilde{KSp}}^{-2}(S^{4n+2})\cong \mathbb{Z}^{n-m+1}$ such that ${\widetilde{KSp}}^{-2}(X)=\mathbb{Z}\{\xi_1,\xi_2,\cdots ,\xi_{n-m+1}\}$, where $\xi_1\in {\widetilde{KSp}}^{-2}(S^{4m+2})$, $\xi_2\in {\widetilde{KSp}}^{-2}(S^{4m+6})$, $\cdots$, and also $\xi_{n-m+1}\in {\widetilde{KSp}}^{-2}(S^{4n+2})$,\\
$(b)$ there is an isomorphism ${\tilde{K}}^{-2}(X)\rightarrow {\tilde{K}}^{-2}(S^{4m+2}) \oplus \cdots \oplus {\tilde{K}}^{-2}(S^{4n+2}) \cong \mathbb{Z}^{n-m+1}$ such that ${\tilde{K}}^{-2}(X)=\mathbb{Z}\{\xi'_1,\xi'_2,\cdots ,\xi'_{n-m+1}\}$, where $\xi'_1\in {\tilde{K}}^{-2}(S^{4m+2})$, $\xi'_2\in {\tilde{K}}^{-2}(S^{4m+6})$, $\cdots$, and also $\xi'_{n-m+1}\in {\tilde{K}}^{-2}(S^{4n+2})$,\\
$(c):$ if $m$ is even integer then
\begin{displaymath}
c'(\xi_1)=\xi'_1,\quad c'(\xi_2)=2\xi'_2,\quad \cdots, \quad c'(\xi_{n-m+1}) = \left\{ \begin{array}{ll}
\xi'_{n-m+1} & \textrm{if $n$ is even },\\
2\xi'_{n-m+1} & \textrm{if $n$ is odd },\\
\end{array} \right.
\end{displaymath}
and if $m$ is odd integer then
\begin{displaymath}
c'(\xi_1)=2\xi'_1,\quad c'(\xi_2)=\xi'_2,\quad \cdots, \quad c'(\xi_{n-m+1}) = \left\{ \begin{array}{ll}
\xi'_{n-m+1} & \textrm{if $n$ is even },\\
2\xi'_{n-m+1} & \textrm{if $n$ is odd }.\\
\end{array} \right.
\end{displaymath}
\end{lem}
\begin{proof}
First, we consider the case $Q_2$. Put $L_1={\widetilde{KSp}}^{-2}(S^{4m+3})$ and $L_2={\widetilde{KSp}}^{-2}(S^{4m+5})$. The cofibration sequence
$$S^{4m+2}\rightarrow S^{4m-1}\wedge Q_2\rightarrow S^{4m+6}$$
induces the following commutative diagram of exact sequences
\begin{equation*}
\begin{tikzpicture}[baseline=(current bounding box.center)]
\matrix(m)[matrix of math nodes,
 row sep=2em, column sep=2em,
 text height=1.5ex, text depth=0.20ex]
  {L_1&{\widetilde{KSp}}^{-2}(S^{4m+6})&{\widetilde{KSp}}^{-2}(S^{4m-1}\wedge Q_2)&{\widetilde{KSp}}^{-2}(S^{{4m+2}})&L_2\\0&{\tilde{K}}^{-2}(S^{4m+6})&{\tilde{K}}^{-2}(S^{4m-1}\wedge Q_2)&{\tilde{K}}^{-2}(S^{{4m+2}})&0,\\}; \path[->,font=\scriptsize]
  (m-1-1) edge (m-1-2)
  (m-1-2) edge (m-1-3)
  (m-1-3) edge (m-1-4)
  (m-1-4) edge (m-1-5)
  (m-1-2) edge node[left] {$c'=\sigma_j$}(m-2-2)
  (m-1-3) edge node[left] {$c'$}(m-2-3)
  (m-1-4) edge node[left] {$c'=\sigma_i$}(m-2-4)
      (m-2-1) edge (m-2-2)
  (m-2-2) edge (m-2-3)
  (m-2-3) edge (m-2-4)
  (m-2-4) edge (m-2-5);
 \end{tikzpicture}
 \end{equation*}
where $\sigma_i=1, \sigma_j=2 $ when $m$ is even and $\sigma_i=2, \sigma_j=1$ when $m$ is odd, (for example, see \cite{[KK]}). Note that there is an isomorphism ${\tilde{K}}^{-2}(S^{2i})\cong\mathbb{Z}$. We know that $L_1$ is isomorphic to zero if $m$ is even and is isomorphic to $\mathbb{Z}_2$ if $m$ is odd. Also, we have that $L_2$ is zero. So when $m$ is odd then we get the following exact sequence
$$\mathbb{Z}_2 \rightarrow \mathbb{Z} \rightarrow {\widetilde{KSp}}^{-2}(S^{4m-1}\wedge Q_2)\rightarrow \mathbb{Z} \rightarrow 0.$$
Thus by the exactness we get ${\widetilde{KSp}}^{-2}(S^{4m-1}\wedge Q_2)$ is isomorphic to $\mathbb{Z}\oplus \mathbb{Z}$. When $m$ is even then we get the following exact sequence
$$0\rightarrow \mathbb{Z} \rightarrow {\widetilde{KSp}}^{-2}(S^{4m-1}\wedge Q_2)\rightarrow \mathbb{Z} \rightarrow 0.$$
Thus this exact sequence splits and we get ${\widetilde{KSp}}^{-2}(S^{4m-1}\wedge Q_2)$ is isomorphic to $\mathbb{Z}\oplus \mathbb{Z}$. Note that in both cases ${\tilde{K}}^{-2}(S^{4m-1}\wedge Q_2)$ is a free abelian group isomorphic to $\mathbb{Z}\oplus \mathbb{Z}$. Therefore we have
$${\widetilde{KSp}}^{-2}(S^{4m-1}\wedge Q_2)=\mathbb{Z}\{\xi_1,\xi_2\}, \quad {\tilde{K}}^{-2}(S^{4m-1}\wedge Q_2)=\mathbb{Z}\{\xi'_1,\xi'_2\},$$
where $\xi_1\in {\widetilde{KSp}}^{-2}(S^{4m+2})$, $\xi_2\in {\widetilde{KSp}}^{-2}(S^{4m+6})$, $\xi'_1\in {\tilde{K}}^{-2}(S^{4m+2})$ and $\xi'_2\in {\tilde{K}}^{-2}(S^{4m+6})$.\par
Put $L_3={\widetilde{KSp}}^{-2}(S^{4m}\wedge Q_{n-m})$. There is a cofibration sequence
$$S^{4n+1}\overset{\theta}\rightarrow S^{4m-1}\wedge Q_{n-m} \rightarrow X\rightarrow S^{4n+2}\overset{\Sigma\theta}\rightarrow S^{4m} \wedge Q_{n-m}.$$
This sequence induces the following commutative diagram of exact sequences
\begin{equation*}
\begin{tikzpicture}[baseline=(current bounding box.center)]
\matrix(m)[matrix of math nodes,
 row sep=2em, column sep=2em,
 text height=1.5ex, text depth=0.20ex]
  {L_3&{\widetilde{KSp}}^{-2}(S^{4n+2})&{\widetilde{KSp}}^{-2}(X)&{\widetilde{KSp}}^{-2}(S^{4m-1}\wedge Q_{n-m})&0\\0&{\tilde{K}}^{-2}(S^{4n+2})&{\tilde{K}}^{-2}(X)&{\tilde{K}}^{-2}(S^{4m-1}\wedge Q_{n-m})&0, \quad (\ast)\\}; \path[->,font=\scriptsize]
  (m-1-2) edge (m-1-3)
  (m-1-3) edge (m-1-4)
  (m-1-4) edge (m-1-5)
  (m-1-1) edge node[above] {$(\Sigma\theta)^*$} (m-1-2)
  (m-1-2) edge node[left] {$c'=\sigma_k$}(m-2-2)
  (m-1-3) edge node[left] {$c'$}(m-2-3)
  (m-1-4) edge node[left] {$c'$}(m-2-4)
      (m-2-1) edge (m-2-2)
  (m-2-2) edge (m-2-3)
  (m-2-3) edge (m-2-4)
  (m-2-4) edge (m-2-5);
 \end{tikzpicture}
 \end{equation*}
where $\sigma_k=1$ when $n$ is even and $\sigma_k=2$ when $n$ is odd. Now, since by Lemma \ref{e2} the group $L_3$ is isomorphic to a finite cyclic group, the map $(\Sigma\theta)^*$ is trivial. Thus we can rewrite the upper row of the exact sequence in diagram $(\ast)$ as the following exact sequence
\begin{align*}
0\rightarrow {\widetilde{KSp}}^{-2}(S^{4n+2})\rightarrow {\widetilde{KSp}}^{-2}(X)\rightarrow {\widetilde{KSp}}^{-2}(S^{4m-1}\wedge Q_{n-m})\rightarrow 0.
\end{align*}
Therefore the upper and the down rows of the exact sequences in diagram $(\ast)$ splits. Thus by induction, we can conclude that ${\widetilde{KSp}}^{-2}(X)$ and ${\tilde{K}}^{-2}(X)$ are free abelian groups that are isomorphic to $\mathbb{Z}^{n-m+1}$ with their basis $\xi_1,\xi_2,\cdots ,\xi_{n-m+1}$ and $\xi'_1,\xi'_2,\cdots ,\xi'_{n-m+1}$, respectively, where $\xi_{n-m+1}\in {\widetilde{KSp}}^{-2}(S^{4n+2})$ and $\xi'_{n-m+1}\in {\tilde{K}}^{-2}(S^{4n+2})$. Now according to the definition of the maps $c'$, we can choose $\xi_1$, $\xi'_1$, $\xi_2$, $\xi'_2$, $\cdots$, $\xi_{n-m+1}$ and $\xi'_{n-m+1}$ such that if $m$ is even integer then we have
\begin{displaymath}
c'(\xi_1)=\xi'_1,\quad c'(\xi_2)=2\xi'_2,\quad \cdots, \quad c'(\xi_{n-m+1}) = \left\{ \begin{array}{ll}
\xi'_{n-m+1} & \textrm{if $n$ is even },\\
2\xi'_{n-m+1} & \textrm{if $n$ is odd },\\
\end{array} \right.
\end{displaymath}
and if $m$ is odd integer then we have
\begin{displaymath}
c'(\xi_1)=2\xi'_1,\quad c'(\xi_2)=\xi'_2,\quad \cdots, \quad c'(\xi_{n-m+1}) = \left\{ \begin{array}{ll}
\xi'_{n-m+1} & \textrm{if $n$ is even },\\
2\xi'_{n-m+1} & \textrm{if $n$ is odd }.
\end{array} \right.
\end{displaymath}
\end{proof}
Consider the map $c'\colon Sp(n-m+1)\rightarrow SU(2n-2m+2)$. By restriction of the map $c'$ to their quasi-projective spaces, we get to map ${\bar{c}}'\colon Q_{n-m+1}\rightarrow \Sigma\mathbb{C}P^{2n-2m+1}$. The cohomologies of $Q_{n-m+1}$ and $\Sigma\mathbb{C}P^{2n-2m+1}$ are given by
\begin{align*}
&H^* (Q_{n-m+1})=\mathbb{Z}\{{\bar{y}}_3,{\bar{y}}_7, \cdots ,{\bar{y}}_{4n-4m+3}\},\\
&H^* (\Sigma\mathbb{C}P^{2n-2m+1})=\mathbb{Z}\{{\bar{x}}_3,{\bar{x}}_5, {\bar{x}}_7,\cdots, {\bar{x}}_{4n-4m+3}\},
\end{align*}
such that ${\bar{c}}'({\bar{x}}_3)={\bar{y}}_3$, ${\bar{c}}'({\bar{x}}_5)=0$, $ {\bar{c}}'({\bar{x}}_7)={\bar{y}}_7$, $\cdots$, and ${\bar{c}}'({\bar{x}}_{4n-4m+3})={\bar{y}}_{4n-4m+3}$. We denote a generator of $\tilde{K}(S^{2n})\cong \mathbb{Z}$ by $\zeta_n$. Recall that $H^*(\mathbb{C}P^{2n-2m+1})=\mathbb{Z}[t]/(t^{2n-2m+2})$ and $K(\mathbb{C}P^{2n-2m+1})=\mathbb{Z}[x]/(x^{2n-2m+2})$, where $|t|=2$. Note that ${\tilde{K}}^{-2}(\Sigma^{4m}\mathbb{C}P^{2n-2m+1})\cong {\tilde{K}}^{0}(\Sigma^{4m+2}\mathbb{C}P^{2n-2m+1})$ is a free abelian group generated by $\zeta_{2m+1}\otimes x$, $\zeta_{2m+1}\otimes x^2$, $\cdots$, $\zeta_{2m+1}\otimes x^{2n-2m+1}$, with the following Chern characters
\begin{align*}
&ch_{2n+1}(\zeta_{2m}\otimes x)=ch_{2m}(\zeta_{2m})ch_{2n-2m+1}(x)=\frac{1}{(2n-2m+1)!}\sigma^{4m}t^{2n-2m+1},\\\\
&ch_{2n+1}(\zeta_{2m}\otimes x^2)=ch_{2m}(\zeta_{2m})ch_{2n-2m+1}(x^2)=A\sigma^{4m},\\\\
&ch_{2n+1}(\zeta_{2m}\otimes x^3)=ch_{2m}(\zeta_{2m})ch_{2n-2m+1}(x^3)=B\sigma^{4m},\\
\vdots\\
&ch_{2n+1}(\zeta_{2m}\otimes x^{2n-2m+1})=ch_{2m}(\zeta_{2m})ch_{2n-2m+1}(x^{2n-2m+1})=C\sigma^{4m},
\end{align*}
where $A$, $B$ and $C$ are equal to
$$A= ch_{2n-2m+1}(x^2)= \sum_{\substack {i_1+i_2=2n-2m+1,\\ 1\leq i_1\leq n-m}}ch_{i_1}x ch_{i_2}x =\sum_{k=1}^{n-m}\frac{1}{k!(2n-2m+1-k)!}t^{2n-2m+1},$$
\begin{align*}
B= ch_{2n-2m+1}(x^3)=&ch_1x\sum_{\substack {i_1+i_2=2n-2m,\\ 1\leq i_1\leq i_2}}ch_{i_1}x ch_{i_2}x+ch_2 x\sum_{\substack {i_1+i_2=2n-2m-1,\\ 2\leq i_1\leq i_2}}ch_{i_1}x ch_{i_2}x+\cdots \\
&+ ch_{n-m-1}x\sum_{\substack {i_1+i_2=n-m+2,\\ n-m-1\leq i_1\leq i_2}}ch_{i_1}x ch_{i_2}x  + \sum_{\substack {i_1+i_2=2n-2m+1, \\ i_1, i_2\geq 1}}ch_{i_1}x ch_{i_2}x^2,
\end{align*}
and
\begin{align*}
C= ch_{2n-2m+1}(x^{2n-2m+1})=&ch_1x\sum_{\substack {i_1+\cdots + i_{2n-2m}=2n-2m,\\ 1\leq i_1\leq i_2 \leq \cdots \leq i_{2n-2m}}} ch_{i_1}x^{i_1}\cdots ch_{i_{2n-2m}}x^{i_{2n-2m}} \\&+ch_2x^2\sum_{\substack {i_1+\cdots + i_{n-m-1}=2n-2m-1,\\ 2\leq i_1\leq i_2\leq \cdots \leq i_{n-m-1}}}ch_{i_1}x^{i_1} \cdots ch_{i_{n-m-1}}x^{i_{n-m-1}} \\
&+ch_3x^3 \sum_{\substack {i_1+\cdots +i_{n-m-2}=2n-2m-2,\\ 3\leq i_1\leq i_2\leq \cdots \leq i_{n-m-2}}}ch_{i_1}x^{i_1}\cdots  ch_{i_{n-m-2}}x^{i_{n-m-2}} +\cdots \\&+ ch_{n-m}x^{n-m}\sum_{\substack {i_1=n-m+1,\\ i_1\geq n-m}}ch_{i_1}x^{i_1}.
\end{align*}
Consider the map ${\bar{c}}'\colon {\tilde{K}}^{-2}(\Sigma^{4m}\mathbb{C}P^{2n-2m+1})\rightarrow {\tilde{K}}^{-2}(S^{4m-1}\wedge Q_{n-m+1})$, we can put $\xi'_1$, $\xi'_2$, $\cdots$ and  $\xi'_{n-m+1}$ such that
$\xi'_1={\bar{c}}'(\zeta_{2m}\otimes x)$, $\xi'_2={\bar{c}}'(\zeta_{2m}\otimes x^3)$, $\cdots$, and $\xi'_{n-m+1}={\bar{c}}'(\zeta_{2m}\otimes x^{2n-2m+1})$. We have the following proposition.
\begin{prop}\label{e4}
The image of $\psi$ is generated by
\begin{displaymath}
\left\{ \begin{array}{ll}
\frac{(2n+1)!}{(2n-2m+1)!} \sigma^{4m-1}\otimes \bar{y}_{4n-4m+3} &  \textrm{if $m$ is even },\\\\
\frac{2(2n+1)!}{(2n-2m+1)!} \sigma^{4m-1}\otimes \bar{y}_{4n-4m+3} &  \textrm{if $m$ is odd }.\\
\end{array} \right.
\end{displaymath}
\end{prop}
\begin{proof}
Consider the following commutative diagram
\begin{equation}
\begin{tikzpicture}[baseline=(current bounding box.center)]
\matrix(m)[matrix of math nodes,
 row sep=2em, column sep=2em,
 text height=1.5ex, text depth=0.20ex]
 {{\tilde{K}}^{-2}(\Sigma^{4m}\mathbb{C}P^{2n-2m+1})& H^{4n+2}(\Sigma^{4m}\mathbb{C}P^{2n-2m+1})\\ {\tilde{K}}^{-2}(S^{4m-1}\wedge Q_{n-m+1})& H^{4n+2}(S^{4m-1}\wedge Q_{n-m+1}),\\}; \path[->,font=\scriptsize]
  (m-1-1) edge node[above] {$\varphi'$}(m-1-2)
  (m-1-1) edge node[left] {${\bar{c}}'$} (m-2-1)
  (m-1-2) edge node[right] {$\cong$}(m-2-2)
      (m-2-1) edge node[above] {$\varphi$} (m-2-2);
 \end{tikzpicture}
\end{equation}
where the map $\varphi'$ is defined similarly to the map $\varphi$. By definition of the map $\varphi'$, we have
\begin{align*}
&\varphi'(\zeta_{2m}\otimes x)=(2n+1)!ch_{2n+1}(\zeta_{2m}\otimes x)=
\frac{(2n+1)!}{(2n-2m+1)!}\sigma^{4m}t^{2n-2m+1},\\\\
&\varphi'(\zeta_{2m}\otimes x^2)=(2n+1)!ch_{2n+1}(\zeta_{2m}\otimes x^2)=
A(2n+1)!\sigma^{4m},\\\\
&\varphi'(\zeta_{2m}\otimes x^3)=(2n+1)!ch_{2n+1}(\zeta_{2m}\otimes x^3)=
B(2n+1)!\sigma^{4m},\\
\vdots\\
&\varphi'(\zeta_{2m}\otimes x^{2n-2m+1})=(2n+1)!ch_{2n+1}(\zeta_{2m}\otimes x^{2n-2m+1})=
C(2n+1)!\sigma^{4m}.
\end{align*}
Therefore according to the commutativity of diagram $(5.3)$, we have
\begin{align*}
&\varphi(\xi'_1)=\varphi'(\zeta_{2m}\otimes x)=
\frac{(2n+1)!}{(2n-2m+1)!}\sigma^{4m-1}\otimes \bar{y}_{4n-4m+3},\\\\
&\varphi(\xi'_2)=\varphi'(\zeta_{2m}\otimes x^3)=B'(2n+1)!\sigma^{4m-1}\otimes \bar{y}_{4n-4m+3},\\
\vdots\\
&\varphi(\xi'_{n-m+1})=\varphi'(\zeta_{2m}\otimes x^{2n-2m+1})=C'(2n+1)!\sigma^{4m-1}\otimes \bar{y}_{4n-4m+3},
\end{align*}
where $B=B't^{2n-2m+1}$ and $C=C't^{2n-2m+1}$. Thus by the commutativity of diagram $(5.2)$, we get
\begin{displaymath}
\psi(\xi_1)= \varphi(c'(\xi_1))=\pm \left\{ \begin{array}{ll}
\frac{(2n+1)!}{(2n-2m+1)!} \sigma^{4m-1}\otimes \bar{y}_{4n-4m+3} & \textrm{if $m$ is even },\\\\
\frac{2(2n+1)!}{(2n-2m+1)!} \sigma^{4m-1}\otimes \bar{y}_{4n-4m+3} & \textrm{if $m$ is odd },\\
\end{array} \right.
\end{displaymath}
\begin{displaymath}
\psi(\xi_2)= \varphi(c'(\xi_2))=\pm \left\{ \begin{array}{ll}
2B'(2n+1)! \sigma^{4m-1}\otimes \bar{y}_{4n-4m+3} & \textrm{if $m$ is even },\\\\
B'(2n+1)! \sigma^{4m-1}\otimes \bar{y}_{4n-4m+3} & \textrm{if $m$ is odd },\\
\end{array} \right.
\end{displaymath}
$\qquad \vdots$
\begin{displaymath}
\psi(\xi_{n-m+1})= \varphi(c'(\xi_{n-m+1}))= \left\{ \begin{array}{ll}
-(2n+1)! C' \sigma^{4m-1}\otimes \bar{y}_{4n-4m+3} & \textrm{if $n$ is even },\\\\
 2(2n+1)! C' \sigma^{4m-1}\otimes \bar{y}_{4n-4m+3} & \textrm{if $n$ is odd }.\\
\end{array} \right.
\end{displaymath}
Since the coefficients $B',\cdots, C'$ are divisible by $\frac{1}{(2n-2m+1)!}$, the image of $\psi$ is included in the submodule generated by $\frac{(2n+1)!}{(2n-2m+1)!} \sigma^{4m-1}\otimes \bar{y}_{4n-4m+3}$ if $m$ is even and is included in the submodule generated by $\frac{2(2n+1)!}{(2n-2m+1)!} \sigma^{4m-1}\otimes \bar{y}_{4n-4m+3}$ if $m$ is odd. Thus we can conclude that
\begin{displaymath}
Im \psi\cong \left\{ \begin{array}{ll}
\mathbb{Z}\{\frac{(2n+1)!}{(2n-2m+1)!} \sigma^{4m-1}\otimes \bar{y}_{4n-4m+3}\} & \textrm{if $m$ is even },\\\\
\mathbb{Z}\{\frac{2(2n+1)!}{(2n-2m+1)!} \sigma^{4m-1}\otimes \bar{y}_{4n-4m+3}\} & \textrm{if $m$ is odd }.
\end{array} \right.
\end{displaymath}
\end{proof}
Therefore by Lemma \ref{e1} and Proposition \ref{e4}, we get the following theorem.
\begin{thm}\label{e5}
There is an isomorphism
\begin{displaymath}
[X, Sp(n)]\cong \left\{ \begin{array}{ll}
\mathbb{Z}_{\frac{(2n+1)!}{(2n-2m+1)!}} & \textrm{if $m$ is even },\\\\
\mathbb{Z}_{\frac{2(2n+1)!}{(2n-2m+1)!}} & \textrm{if $m$ is odd }. \quad \Box
\end{array} \right.
\end{displaymath}
\end{thm}
\textbf{Proof of Theorem \ref{a5}}\\
We recall that $X=S^{4m-1}\wedge Q_{n-m+1}$. In the following theorem that was proved in \cite{[N1]}, we obtain the commutators in the group $[X,Sp(n)]$.
\begin{thm}\label{e6}
Let $X$ be a $CW$-complex with dim $X \leq 4n+2$. For $\alpha,\beta\in Sp_n(X)$, the commutator $[\alpha, \beta]$ in $Sp_n(X)$ is given as
$$[\alpha, \beta]=\iota([\sum_{i+j=n+1}\alpha^*(y_{4i-1})\beta^*(y_{4j-1})]). \quad \Box$$
\end{thm}
Put $X'=S^{4m-1}\times Q_{n-m+1}$. Consider the following compositions
$$S^{4m-1}\times Q_{n-m+1} \overset{p_1}\longrightarrow S^{4m-1} \overset{\varepsilon_{m,n}}\longrightarrow Sp(n), $$
$$S^{4m-1}\times Q_{n-m+1}\overset{p_2}\longrightarrow Q_{n-m+1} \overset{\epsilon_{m,n}}\longrightarrow Sp(n), $$
where the maps $p_1$ and $p_2$ are the first and the second projections. Put $\varepsilon_{m,n} \circ p_1=\alpha$ and $\epsilon_{m,n} \circ p_2=\beta$. The commutator $[\alpha, \beta]\in [X', Sp(n)]$ is the following composition
$$X' \overset{\bar{\Delta}}\longrightarrow X'\wedge X'  \overset{\alpha \wedge \beta}\longrightarrow Sp(n)\wedge Sp(n)\overset{\gamma}\longrightarrow Sp(n),$$
where $\bar{\Delta}$ and $\gamma$ are the reduced diagonal map and the commutator map. Consider the following diagram
\begin{equation*}
\begin{tikzpicture}[baseline=(current bounding box.center)]
\matrix(m)[matrix of math nodes,
 row sep=2em, column sep=2em,
 text height=1.5ex, text depth=0.20ex]
 {X'& X'\times X'& X'& X& Sp(n)\wedge Sp(n)& Sp(n)\\ & X'\wedge X'\\}; \path[->,font=\scriptsize]
  (m-1-1) edge node[above] {$\Delta$}(m-1-2)
  (m-1-2) edge node[above] {$p_1\times p_2$} (m-1-3)
  (m-1-3) edge node[above] {$q$}(m-1-4)
      (m-1-4) edge node[above] {$\varepsilon_{m,n}\wedge\epsilon_{m,n}$} (m-1-5)
  (m-1-5) edge node[above] {$\gamma$} (m-1-6)
  (m-1-1) edge node[left] {$\bar{\Delta}$}(m-2-2)
  (m-2-2) edge node[right] {$\quad\alpha \wedge \beta$}(m-1-5);
 \end{tikzpicture}
\end{equation*}
where $\Delta$ is the diagonal map and $q:X'\rightarrow X$ is the quotient map. Since the composition
$$X'=S^{4m-1}\times Q_{n-m+1} \overset{\Delta}\longrightarrow X'\times X'  \overset{p_1\times p_2}\longrightarrow S^{4m-1}\times Q_{n-m+1}=X' $$
is $1$, the Samelson product $\langle \varepsilon_{m,n}, \epsilon_{m,n}\rangle = \gamma \circ (\varepsilon_{m,n}\wedge\epsilon_{m,n})$ is equal to the commutator $[\alpha,\beta]$ in the
group $[X, Sp(n)]$. So by Theorem \ref{e6}, we have
$$[\alpha,\beta ]=\iota ([{\varepsilon_{m,n}}^*(y_{4m-1})\otimes {\epsilon_{m,n}}^*(y_{4n-4m+3})])= \iota([\sigma^{4m-1}\otimes \bar{y}_{4n-4m+3}]).$$
Therefore by Theorem \ref{e5}, when $m$ is even then the order of the Samelson product $\langle \varepsilon_{m,n},\epsilon_{m,n} \rangle$ is equal to $\frac{(2n+1)!}{(2n-2m+1)!}$ and when $m$ is odd then the order of the Samelson product $\langle \varepsilon_{m,n},\epsilon_{m,n} \rangle$ is equal to $\frac{2(2n+1)!}{(2n-2m+1)!}$. $\quad \Box$\\


 School of Mathematics, Institute for Research in Fundamental Sciences (IPM), P.O. Box 19395-5746, Tehran, Iran\\
E-mail: sj.mohammadi@ipm.ir\\

\begin{thebibliography}{WWW99}

{\small\small

\bibitem{[B]} {} R. Bott, A note on the Samelson product in the classical groups, Commentarii Mathematici Helvetici. \textbf{34} (1960), 249-256.

\bibitem{[BK]} {} A. K. Bousfield and D. M. Kan, Homotopy limits, completions and localizations, Lecture Notes in Math. \textbf{304} (1972), Springer-Verlag, Berlin-NewYork.

\bibitem{[CN]} {} F.R. Cohen and J.A. Neisendorfer, A construction of $p$-local H-spaces, Algebraic topology, Aarhus 1982, Lecture Notes in Math., vol. 1051, pp. 351-359, Springer, Berlin, 1984.

\bibitem{[G]} {}  T. Ganea, Cogroups and suspensions, Invent. Math. \textbf{9} (1970), 185-197.

\bibitem{[H1]} {} H. Hamanaka, On Samelson products in $p$-localized unitary groups, Topology Appl. \textbf{154} (2017), 573-583.
\bibitem{[H2]} {} H. Hamanaka, On $[X, U(n)]$ when dim $X$ is $2n+1$, J. Math. Kyoto Univ. \textbf{44 (3)}
(2004), 655-667.

\bibitem{[H3]} {}  H. Hamanaka and A. Kono, On $[X, U(n)]$ when dim $X=2n$, J. Math. Kyoto Univ.
\textbf{43 (2)} (2003), 333-348.

\bibitem{[HKK]} {}   H. Hamanaka, S. Kaji  and  A. Kono, Samelson products in $Sp(2)$, Topology and its Applications. \textbf{155} (2008), 1207-1212.

\bibitem{[J]} {} I. M. James, Spaces associated with Stiefel manifolds, Proc. London Math. Soc. \textbf{9} (1959), 115-140.

\bibitem{[JW]} {}  I. James and J. Whitehead, The homotopy theory of Sphere-bundles over spheres I,
Proc. Lond. Math. Soc. \textbf{4} (1954), 196-218.

 \bibitem{[KK]} {}  D. Kishimoto and A. Kono, On the homotopy types of $Sp(n)$ gauge groups, Algebraic and Geometric Topology. \textbf{19} (2019), 491-502.

 \bibitem{[KKKT]} {} Y. Kamiyama, D. Kishimoto, A. Kono and S. Tsukuda, Samelson products of $SO(3)$ and
                 applications, Glasg. Math. J. \textbf{49} (2007), 405-409.

\bibitem{[KKT]} {}   D. Kishimoto, A. Kono and M. Tsutaya, On $p$-local homotopy types of gauge groups,
Proc. Roy. Soc. Edinburgh. A \textbf{144} (2014), 149-160.

\bibitem{[KT]} {}  A. Kono and S. Theriault, The order of the commutator on $SU(3)$ and an application
to gaug groups, Bull. Belg. Math. Soc. Simon Stevin. \textbf{20} (2013), 359-370.

\bibitem{[M]} {} C. A. McGibbon, Homotopy commutativity in localized groups, Amer. J. Math. \textbf{106}
(1984), 665-687.

\bibitem{[MNT]} {}  M. Mimura, G. Nishida and H. Toda, Mod $p$ decomposition of compact Lie groups,
Publ. Res. Inst. Math. Sci. \textbf{13} (1977), 627-680.

 \bibitem{[N]} {} T. Nagao, On Samelson products in $Sp(n)$, J. Math. Kyoto Univ. \textbf{49-1} (2009), 225-234.

 \bibitem{[N1]} {} T. Nagao, On the groups $[X,Sp(n)]$ with dim $ X\leq 4n+2$, Kyoto J. Math.
\textbf{48} (2008), 149-165.

 \bibitem{[O]} {} H. Oshima, Samelson products in the exceptional Lie group of rank $2$, J. Math. Kyoto Univ.
\textbf{45} (2005), 411-420.

\bibitem{[T]} {} S. Theriault, Odd primary homotopy types of $SU(n)$-gauge groups, Algebr.
Geom.Topol. \textbf{17} (2017), 1131-1150.

\bibitem{[T1]} {}  H. Toda, Composition Methods in Homotopy Groups of Spheres, Annals of Mathematics Studies 49, Princeton University Press, Princeton N.J., 1962.

}
\end{thebibliography}
\end{document}